\documentclass[10pt,reqno,final]{amsart}
\usepackage[notcite,notref]{showkeys}
\usepackage{url,hyperref,multirow}
\usepackage{color}
\usepackage{epsfig,subfigure,amssymb,amsmath,version,slashbox}
\usepackage{amssymb,version,graphicx,fancybox,mathrsfs,pifont,booktabs}

\catcode`\@=11 \theoremstyle{plain}

\renewcommand{\theequation}{\thesection.\arabic{equation}}
\newtheorem{lemma}{Lemma}[section]
\newtheorem{theorem}{Theorem}[section]
\newtheorem{corollary}{Corollary}[section]
\newtheorem{proposition}{Proposition}[section]

\newtheorem{remark}{Remark}[section]

\@addtoreset{figure}{section}
\renewcommand\thefigure{\thesection.\@arabic\c@figure}
\@addtoreset{table}{section}
\renewcommand\thetable{\thesection.\@arabic\c@table}

\newcommand{\bs}[1]{\boldsymbol{#1}}
\def \ri {{\rm i}}

\DeclareSymbolFont{ugmL}{OMX}{mdugm}{m}{n}
\SetSymbolFont{ugmL}{bold}{OMX}{mdugm}{b}{n}

\DeclareMathAccent{\wideparen}{\mathord}{ugmL}{"F3}

\begin{document}
\bibliographystyle{plain}
\graphicspath{{./figs/}}
\baselineskip 13pt

\title[Wave Equation With Exact Nonreflecting Boundary Conditions]
{Fast and Accurate Computation of Time-Domain Acoustic Scattering Problems with Exact Nonreflecting Boundary Conditions}
\author[L. Wang, ~ B. Wang ~  $\&$~ X. Zhao~~] {Li-Lian Wang${}^{1}$, \quad
  Bo Wang${}^{2}$ \quad
and \quad    Xiaodan Zhao${}^{1}$}
\thanks{
\noindent${}^{1}$ Division of Mathematical Sciences, School of Physical
and Mathematical Sciences,  Nanyang Technological University,
637371, Singapore. The research of the authors is partially
supported by Singapore AcRF Tier 1 Grant RG58/08.\\
\indent${}^{2}$  College of Mathematics and Computer Science, Hunan Normal University,
Changsha, Hunan 410081, China. This author would like to thank the Division of Mathematical Sciences of Nanyang Technological University for the hospitality during the visit.
}
\date{}

\begin{abstract}  This paper is concerned with fast and accurate computation of  exterior wave equations truncated via exact circular or
spherical  nonreflecting boundary conditions (NRBCs, which are known to be nonlocal in both time and space).
We first derive analytic expressions for the underlying convolution kernels, which allow for a rapid and accurate
evaluation of the convolution with $O(N_t)$ operations over $N_t$ successive time steps. To handle the nonlocality in space,
we introduce the notion of  boundary perturbation,  which enables us to handle general bounded scatters by solving a sequence
of wave equations in  a regular domain. We propose an efficient spectral-Galerkin solver with Newmark's time integration for
the truncated wave equation in the regular domain. We also provide ample numerical results to show high-order accuracy of NRBCs
and efficiency of the proposed scheme.
\end{abstract}

\maketitle

\vspace*{-15pt}

\section{Introduction}

Wave propagation and scattering problems in unbounded media arise from diverse application  areas  such   as acoustics,
aerodynamics, electromagnetics,  antenna design,  oceanography and among others (see, e.g.,
\cite{Givoli1992,shlager1995selective,collins1994inverse}).  Various approaches have been proposed for their numerical studies
that include the boundary element methods
(cf. \cite{CISB91}), infinite element methods (cf. \cite{Dem.G96}),
perfectly matched layers (PML)  (cf. \cite{Bere94}),   nonreflecting
boundary condition  methods (cf. \cite{Kel.G89,Hagstrom99}), and
among others (cf. \cite{martin2006multiple,Taflove05}). An essential
ingredient for the latter approach is to  truncate an  unbounded
domain to a bounded domain by imposing  an exact or approximate
nonreflecting (absorbing or transparent) boundary condition at the
outer artificial boundary, where the NRBC is designed  to prevent
spurious wave reflection from the artificial boundary (cf.  the
review papers \cite{Givoli1992,GivoliDan2004} and the references
therein).  The frequency-domain approaches for e.g., the
time-harmonic Helmholtz problems have been intensively investigated,
while the time-domain simulations, which are capable of capturing
wide-band signals and modeling more general material inhomogeneities
and  nonlinearities (cf. \cite{Alpert02,CohenG2002}), have been
relatively less studied.

Although some types of NRBCs based on different principles have
been proposed (see, e.g.,
\cite{Enguist77,Bayliss80,ting1986exact,Sofronov1992,Sofronov1998,grote1995exact,Hagstrom99,GivoliDan2004}),
a longstanding issue of time-domain computation is the efficient
treatment for  NRBCs that can  scale and integrate well with the
solver for the underlying truncated problem (cf.
\cite{Thompson2000,HuanThompson2000}). In practice, if an accurate
NRBC is imposed, the artificial boundary could be placed as close as
possible to the scatter that can significantly reduce the
computational cost. In this paper, we restrict our attention to the exact
NRBC on the circular or spherical artificial boundary. One major
difficulty lies in that such a NRBC is {\em global} in space and
time in nature,  as it involves the Fourier/spherical harmonic
expressions in space, and  history dependence in time induced by a
convolution.  The convolution kernel, termed as {\em nonreflecting
boundary kernel} (NRBK) in \cite{Alpert}, is the inverse Laplace
transform of an expression that includes the logarithmic derivative
of a modified Bessel function.  The rapid computation of the NRBK and
convolution is of independent interest.
Alpert, Greengard and Hagstrom \cite{Alpert} proposed a rational
approximation of the logarithmic derivative with a least square
implementation, which allows for a reduction of the summation of the
poles from $O(\nu)$ to $O(\log \nu\log \frac 1 \varepsilon)$ (where
$\nu\gg 1$ is the order of the modified Bessel function  and
$\varepsilon$ is a given tolerance),  and a recursive convolution. Jiang
and Greengard \cite{Jiang1,Jiang2} further considered some interesting
applications to Schr$\ddot o$dinger equations in one and two
dimensions.  Li  \cite{LiJing2006} introduced a more accurate low order approximation
of the three-dimensional NRBK at a slightly expensive cost, where the observation that the Laplace transform
of the three-dimensional NRBK is exactly a rational function lies at the heart of this algorithm.
However, in many cases, the expressions are {\em not} rational functions.
For instance, the two-dimensional NRBK also contains the contributions from the brach-cut along the negative real axis (see Theorem \ref{theorem1} below).
 Lubich and Sch{\"a}dle  \cite{Lubistian2002} developed some fast algorithm for the temporal convolution with  $O(N_t\log N_t)$
 operations (over $N_t$ successive time steps) arising from NRBCs with non-rational expressions for other equations
 (e.g., Schr$\ddot o$dinger equations and damped wave equations).

In this paper, we derive an analytic  formula for the NRBK based on a direct
 inversion of the Laplace transform by the residue theorem (see Theorem \ref{theorem1} below).
In fact,  Sofronov \cite{Sofronov1998} presented some formulas of
similar type by working on much more complicated  expressions of the
kernel in terms of Tricomi's confluent hypergeometric functions.
We show that with these formulas, we can  evaluate the temporal convolution recursively and rapidly with $O(N_t)$ operations
 and almost without extra memory for the history dependence. Moreover, the analytic expression provides a useful apparatus for
 the stability and convergence analysis. It is worthwhile to remark that  Chen \cite{Chen.ZM2009} reformulated the two-dimensional
 wave problem into a first-order system in time and showed the well-posedness of the truncated problem with an alternative formulation
 of the NRBC.

It is known that the nonlocality of the NRBC in space can be
efficiently handled by Fourier/spherical harmonic expansions when
the scatter is a disk or a ball.   Recently, a systematic approach,
based on the boundary perturbation technique (also  called the {\it
transformed field expansion} (TFE) method (cf. \cite{Nic.R03})), has
been developed in \cite{Nic.S06,Fang.DShen07,Nich.Shen2009} for
time-harmonic Helmholtz equations in exterior domains with general
bounded obstacles, under which the whole algorithm boils down to
solving  a sequence of Helmholtz equations in a 2-D annulus or a 3-D
spherical shell. In this paper, we highlight that this notion can be
extended to time-domain computation, though it has  not been
investigated before as far as we know. In this paper,  we propose an
efficient spectral-Galerkin method with Newmark's time integration
for the truncated wave equations in an annulus or a spherical shell,
and provide ample numerical results to show the efficiency of the
solver and high accuracy of  NRBC from several angles.

The rest of the paper is organized as follows.  In Section 2, we present the formulation of  NRBCs,
and  derive the analytic formulas for  NRBKs.
In Section 3, we present some properties of  NRBK and analyze well-posedness of the truncated wave equation.
In Section 4, we outline the notion of the TFE method and propose an efficient spectral-Galerkin and Newmark's
time integration scheme for the truncated wave problem in regular domains.
We provide ample numerical results in Section 5.

\vspace*{-20pt}
\section{Evaluation of nonreflecting boundary kernels}

In this paper, we consider the time-domain acoustic scattering problem with sound-soft boundary conditions on the bounded obstacle:
\begin{align}
&\partial_t^2 U=c^2\Delta U+F,\quad {\rm in}\;\;  \Omega_\infty:={\mathbb R}^d\setminus \bar D,\;\; t>0,\;\; d=2,3; \label{eq1.1a} \\
&U=U_0,\quad \partial_t U=U_1,\quad {\rm in}\;\;\Omega_\infty, \;\; t=0;  \label{eq1.1b} \\
 & U=G,\quad {\rm on}\;\; \Gamma_D,\;\; t>0;\quad \partial_t U+c\partial_{\bs n} U=o(|\bs x|^{(1-d)/2}),\;\; |{\bs x}|\to \infty,\;\; t>0.    \label{eq1.1c}
\end{align}
Here, $D$ is a bounded obstacle (scatter)  with Lipschitz boundary
$\Gamma_D,$ $c>0$ is a given constant,  and the radiation condition
\eqref{eq1.1c}, where  ${\bs n}={\bs x}/|\bs x|$,  corresponds to
the well-known Sommerfeld radiation condition in the frequency
domain. Assume that the data $F, U_0$ and $U_1$ are compactly supported
in a 2-D disk or  a 3-D ball $B$ of radius $b.$

A common way is to reduce this exterior problem  to  the problem in
a bounded domain by imposing an exact or approximate NRBC at the
artificial boundary $\Gamma_b:=\partial B.$ In what follows, we
shall focus on the wave equation truncated by the exact circular or
spherical  NRBC:
\begin{align}
&\partial_t^2 U=c^2\Delta U+F,\quad {\rm in}\;\;  \Omega:=B\setminus \bar D,\;\; t>0,\;\; d=2,3; \label{eq1.1ab} \\
&U=U_0,\quad \partial_t U=U_1, \quad {\rm in}\;\;\Omega, \;\; t=0; \quad 
 U=G,\quad {\rm on}\;\; \Gamma_D,\;\; t>0; \label{eq1.1ccd}\\
&  \partial_r U=T_d(U),\quad {\rm at}\;\;  r=b,\;\; t>0,  \label{eq1.1ce}
\end{align}
where $T_d(U)$ is the so-called time-domain DtN map.

\subsection{Formulation of $T_d(U)$}\label{tdtnmethd}
We first present the expression of $T_d(U)$ in \eqref{eq1.1ce}, and refer to
e.g.,  \cite{Hagstrom99,Alpert02} (and the original references therein) for the detailed derivation.
It is known that the problem \eqref{eq1.1a}-\eqref{eq1.1c}, exterior to $D=B$ with $F=U_0=U_1\equiv 0$ and
$G=U|_{r=b}$ (i.e., the Dirichlet data taken from the interior problem   \eqref{eq1.1ab}-\eqref{eq1.1ce}), can be solved
analytically by using Laplace transform in time and separation of variables in space in polar coordinate $(r,\phi)$/spherical
coordinate $(r,\theta,\phi)$.
By imposing the continuity of directional derivative with respect to $r$ across the artificial boundary $r=b,$ we obtain
the boundary condition \eqref{eq1.1ce} with
 \begin{equation}\label{GdUadd}
T_d(U)=\begin{cases} \Big(-\dfrac 1 c \dfrac{\partial U}{\partial t}-  \dfrac{U}{2r}\Big)\Big|_{r=b}+
\displaystyle\sum_{|n|=0}^\infty \sigma_n(t)\ast  \widehat U_n(b,t)  e^{\ri n\phi},  & d=2,\\[7pt]
\Big(-\dfrac 1 c \dfrac{\partial U}{\partial t}-  \dfrac{U}{r}\Big)\Big|_{r=b} +  \displaystyle\sum_{n=0}^\infty
\displaystyle \sum_{|m|=0}^n \sigma_{n+1/2}(t)\ast  \widehat U_{nm}(b,t) Y_{n}^m(\theta,\phi), & d=3,
\end{cases}
\end{equation}
where
\begin{equation}\label{kernalf}
\sigma_\nu (t):=\mathcal{L}^{-1}\left[\frac{s}{c}+\frac{1}{2b}+\frac{s}{c}\frac{K_\nu'
(sb/c)}{K_\nu(sb/c)}\right],\quad \nu=n, n+ 1/2.
\end{equation}
Here,  $K_{\nu}$ is the modified Bessel function of the second kind  of order $\nu$ (see, e.g., \cite{Abr.S84,watson}), and
$\mathcal{L}^{-1}[H(s)]$ is the inverse  Laplace transform of a Laplace transformable function $h(t)$ with
\[
H(s)=\mathcal{L}[h(t)](s)=\int_0^\infty e^{-st}h(t)\,dt,\quad s\in {\mathbb C},\;\; {\rm Re}(s)>0.
\]
In \eqref{GdUadd}, $\{Y_{n}^m\}$  are the spherical harmonics, which are orthonormal as defined in \cite{Nedelec}, and
$\{\widehat U_n\}$/$\{\widehat U_{nm}\}$ are the Fourier/spherical harmonic expansion coefficients of $U|_{r=b}.$ Note that the convolution is defined as usual: $(f\ast g)(t)=\int_0^t f(t-\tau)g(\tau) d\tau.$

Alternatively, we can represent $T_d(U)$ by the temporal convolution in terms of the expansion coefficients of $\partial_t U|_{r=b}.$ More precisely, we define
\begin{equation}\label{nnectoon}
\omega_\nu(t):=\omega_\nu (t;d):= -\dfrac {(d-1)c} {2b} + c \displaystyle \int_0^t
\sigma_{\nu} (\tau)d \tau,
\end{equation}
and note that $\omega_\nu'(t)=c\sigma_\nu(t).$ Then, we find from \eqref{GdUadd} and integration by parts that
\begin{equation}\label{GdUaddnew}
T_d(U)=-\dfrac 1 c \dfrac{\partial U}{\partial t}\Big|_{r=b}+\frac 1 c  \begin{cases}
\displaystyle\sum_{|n|=0}^\infty \omega_n(t)\ast  \partial_t \widehat U_n(b,t)  e^{\ri n\phi},  & d=2,\\[7pt]
 \displaystyle\sum_{n=0}^\infty
\displaystyle \sum_{|m|=0}^n \omega_{n+1/2}(t)\ast  \partial_t \widehat U_{nm}(b,t) Y_{n}^m(\theta,\phi), & d=3,
\end{cases}
\end{equation}
where for $d=2,3,$
\begin{equation}\label{kernalw}
\omega_{\nu}(t)=\mathcal{L}^{-1}\left[1-\frac{(d-2)c}{2bs}+\frac{K_\nu'
(sb/c)}{K_\nu(sb/c)}\right](t),\quad \nu=n,n+1/2.
\end{equation}

Hereafter, we term  $\sigma_\nu$ and $\omega_\nu$ as the nonreflecting boundary kernels (NRBKs). 
Since  $K_{-n}(z)=K_{n}(z)$ {\rm(}see Formula {\rm 9.6.6} in \cite{Abr.S84}{\rm)},    it suffices to consider $\omega_n$ and $\sigma_n$ with  $n\ge 0,$ for $d=2.$ 

\begin{remark}\label{Augadd}  In the expressions of $\sigma_\nu$ and $\omega_\nu,$ some terms are added, e.g.,
$s/c$ and $1/(2b)$ in \eqref{kernalf},  for the purpose of removing the singular part from the ratio $K_\nu'/K_\nu$.
Indeed, recall the asymptotic formula
for fixed $\nu\ge 0$ and large $|z|$ {\rm(}see Formula {\rm 9.7.2} of \cite{Abr.S84}{\rm):}
\begin{equation}\label{asympt}
K_\nu (z)\sim \sqrt{\frac \pi {2z}} e^{-z}\Big\{1+\frac {4\nu^2-1}{8z}+O(z^{-2})\Big\},
\end{equation}
for $|{\rm arg} z|<{3\pi}/2,$ and  the recurrence  relation:
\begin{equation}\label{asympta}
z K_\nu'(z)=\nu K_\nu (z)-z K_{\nu+1}(z).
\end{equation}
 One verifies that
\begin{equation}\label{akasymp}
\frac{K_\nu'(z)}{K_\nu(z)}\sim -1-\frac 1 {2z} +O(z^{-2}).  \qed
\end{equation}
\end{remark}

\begin{remark}\label{Augadd3} We find that the use of the NRBK $\omega_\nu$ is more convenient,  if one reformulates \eqref{eq1.1ab}
into a first-order {\rm(}with respect to the time variable{\rm)} system  {\rm(cf. \cite{Chen.ZM2009})},
and it is  more suitable for analysis as well, while the NRBK $\sigma_\nu$ is more appropriate  for computation. \qed 
\end{remark}

We see that NRBCs are global in both time and space.  To solve  the truncated problem \eqref{eq1.1ab}-\eqref{eq1.1ce} efficiently,
we need to  (i) invert  Laplace transform to compute  NRBKs;
(ii)  deal with  temporal convolutions efficiently; and (iii) handle  the nonlocality of the NRBC in space effectively.
The rest of the paper will address these  issues.

\subsection{Evaluation of  the NRBKs}

Our starting point is to invert the Laplace transform via  evaluating  the  Bromwich's contour integral:
 \begin{equation} \label{inverlap}
 \begin{split}
\sigma_\nu (t)&=\frac 1 {2\pi \ri }\int_{\gamma-\infty \ri}^{\gamma +\infty \ri} \Big(\frac{s}{c}+\frac{1}{2b}+\frac{s}{c}\frac{K_\nu^{'}
(sb/c)}{K_\nu(sb/c)}\Big) e^{ts} ds=\frac{c}{2b^2\pi \ri}\int_{\gamma-\infty \ri}^{\gamma+\infty \ri}F_\nu (z)e^{czt/b}dz,
\end{split}
\end{equation}
for $\nu=n, n+1/2$ with $ n\ge 0,$ where
\begin{equation}\label{fffunnn}
F_\nu (z)=z+\frac{1}{2}+z\frac{K'_\nu (z)}{K_\nu(z)},
\end{equation}
and $\gamma$ is the Laplace convergence abscissa, which is a generic constant greater than the real part of any  singularity of $F_\nu(z).$

In order to use the residue theorem to evaluate   \eqref{inverlap},
we need to understand the behavior of the poles of $F_\nu(z),$ i.e.,
the zeros of $K_\nu(z).$
\begin{lemma} \label{lemma2}  Let $\nu\ge 0$ be a real number.
\begin{itemize}
\item[{\rm (i)}] If $z$ is a zero of $K_\nu(z),$ then its complex conjugate $\bar z$ is also a zero. 
Moreover, all complex conjugate pairs of zeros lie in the second and third quadrants with  ${\rm Re}(z)<0.$ 
\item[{\rm (ii)}]    The total number of zeros of $K_{\nu}(z)$  is the even integer nearest to $\nu-{1}/{2}$, if $\nu-{1}/{2}$ is not an integer, or  exactly  $\nu-{1}/{2},$
 if  $\nu-{1}/{2}$ is an integer.
 \item[{\rm (iii)}]  All zeros of $K_n(z)$ and $K_{n+{1}/{2}}(z)$ are
simple, and lie  approximately along the left half of the boundary
of an eye-shaped domain around $z=0$ {\rm(}see Figure
{\rm\ref{Figzerodistr})}.

\end{itemize}
\end{lemma}
\begin{proof} The properties (i) and (ii) can be found from Page 511 of
\cite{watson}.  We now consider  the property
(iii).  As a consequence of (i), it suffices to consider the  zeros
of $K_{\nu}(z)$ in the third
 quadrant and on the negative real axis (i.e., with $-\pi\le {\rm arg }z<-\pi/2$) of the complex plane.  According to  Formula 9.6.4 of \cite{Abr.S84},
 we have the following  relation between $K_\nu(z)$ and the Hankel function of the first kind:
\begin{equation}\label{khrela}
K_{\nu}(z)=\frac{\pi \ri}{2} e^{\frac{1}{2}\nu\pi \ri}H_{\nu}^{(1)}(\ri z),\quad -\pi<{\rm arg} z\leq \frac{\pi}{2},
\end{equation}
which implies that all  zeros of $K_\nu(z)$ in the third quadrant
(i.e.,  with $-\pi<{\rm arg }z<-\pi/2$) are obtained by rotating all
zeros of $H_{\nu}^{(1)}(z)$ in the fourth quadrant (i.e.,  with
$-\pi/2<{\rm arg }z<0$) by an angle $-\pi/2.$ Recall that  the zeros
of  $H^1_\nu(z)$ in the fourth quadrant lie approximately  along the  boundary of an eye-shaped domain around $z=0$ (see
Figure 9.6 and Page 441 of \cite{Abr.S84}), whose   boundary curve
intersects the real axis at $z=n$ and the  imaginary axis at $z=-\ri na$, where $a=\sqrt{t_0^2-1}\approx 0.66274$ and
$t_0\approx 1.19968$ is the positive root of $\coth t=t$.
\end{proof}

For clarity, let $M_\nu$ be the total number of
zeros of $K_\nu(z)$ with $\nu=n, n+1/2,$ that is,
\begin{equation}\label{Mnscan}
M_\nu=\begin{cases}
\text{the largest even integer nearest to $n-1/2$},\;   &{\rm for}\; K_n(z),\\
n, \quad &{\rm for}\; K_{n+1/2}(z).
\end{cases}
\end{equation}
We plot in Figure \ref{Figzerodistr} some samples of zeros of
$K_n(z)$ (and $K_{n+1/2}(z)$ for various $n,$ and visualize
that for a given $n$, the zeros  sit on  the left half boundary of
an eye-shaped domain  that intersects the imaginary axis
approximately at $\pm  n\ri,$ and the negative real axis at $-na$
with $a\approx 0.66274$ (see the dashed coordinate  grids) as
predicted by  Lemma \ref{lemma2} (iii).

\begin{figure}[!ht]
\hspace*{-0.6cm}
\includegraphics[width=0.98\textwidth,height=0.3\textwidth]{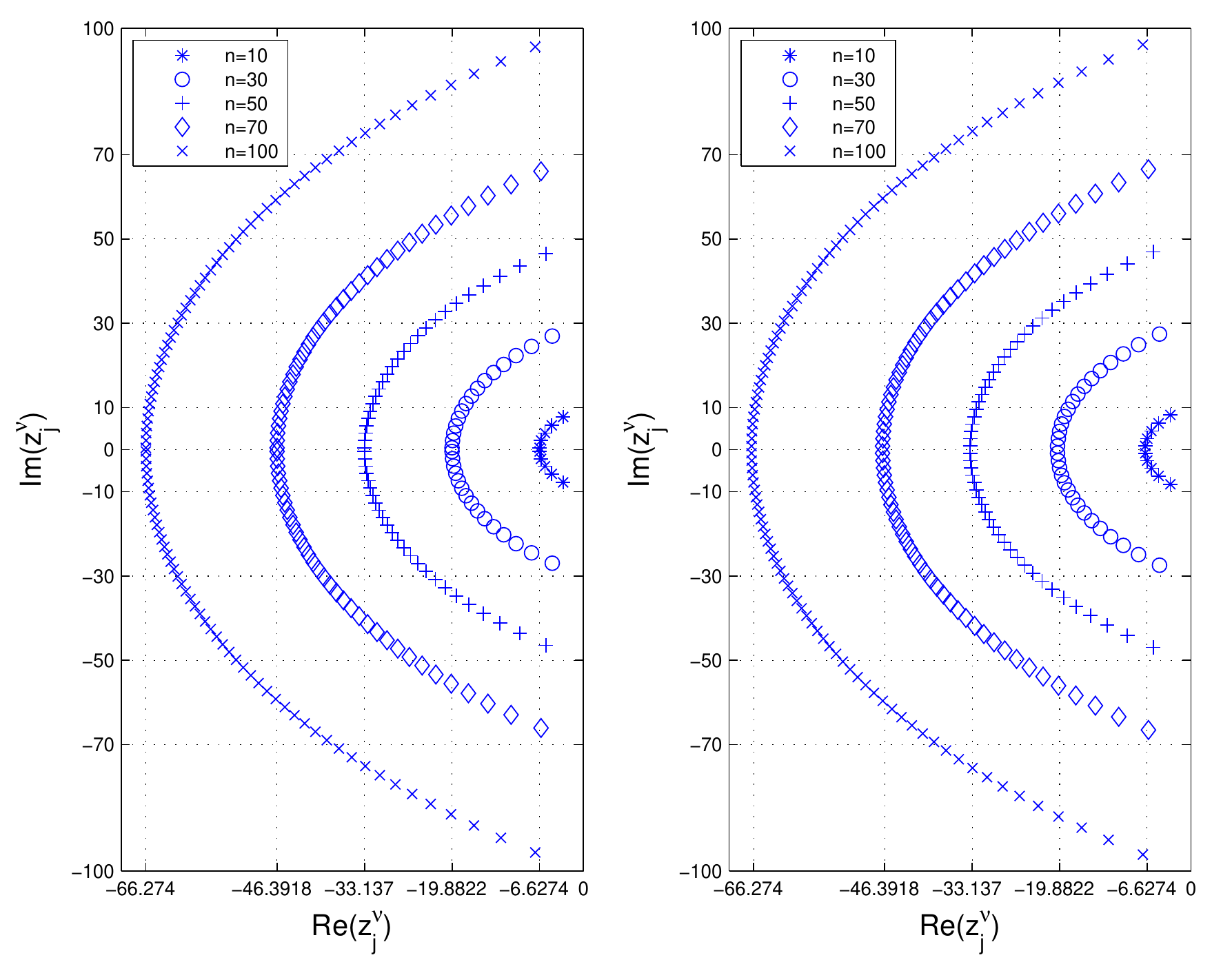}
\caption{\small Distributions of $\{z_j^\nu\}_{j=1}^{M_\nu}$ of
$K_{n}$  (left) and $K_{n+1/2}$ (right)  for various $n$.
}\label{Figzerodistr}
\end{figure}

With the above understanding of  the poles  of the integrand
$F_\nu(z)$ in \eqref{fffunnn}, we now present the exact formula for
the NRBKs $\sigma_\nu(t)$ with $\nu=n, n+1/2.$
\begin{theorem} \label{theorem1}  Let $\nu=n, n+1/2$ with $n\ge 0,$ and
let $\{z_j^\nu\}_{j=1}^{M_\nu}$ be the  zeros  of $K_{\nu}(z).$ Then
\begin{itemize}
\item for $d=2,$
\begin{equation}\label{sigmnt2d}
\sigma_n(t)=\frac{c}{b^2}\bigg\{\sum\limits_{j=1}^{M_n} z_j^n e^{ct z_{j}^{n}/b}+ (-1)^n\int_0^{\infty} \frac{e^{-c
tr/b}}{K^2_{n}(r)+\pi^2I^2_{n}(r)}dr \bigg\},
\end{equation}
\item  for $d=3,$
\begin{equation}\label{sigmnt3d}
\sigma_{\nu}(t)=\frac{c}{b^2}\sum\limits_{j=1}^{M_\nu} z_j^\nu e^{ct z_{j}^{\nu}/b},\quad \nu=n+1/2,
\end{equation}
\end{itemize}
where $I_n(z)$ is the modified Bessel function of the first kind
{\rm (cf. \cite{Abr.S84})}.
\end{theorem}
We sketch the proof  in Appendix \ref{pffssl} by applying the
residue theorem  to the Bromwich's contour integral
\eqref{inverlap}. We remark that Sofronov \cite{Sofronov1998}
derived  formulas of similar type by  working on  much more
complicated expressions  in terms of Tricomi's
confluent hypergeometric functions. However, the formulas in the
above theorem are more compact and informative.

\begin{remark}\label{ksweiomegs}
Based on a delicate study of the logarithmic derivative of the Hankel function $H^{(1)}_\nu(z)$,
Alpert et al. \cite{Alpert} {\rm (}see Theorem  {\rm 4.1} and Lemma {\rm 4.2 in \cite{Alpert})} obtained the following  formula:
\begin{equation}\label{znuschform}
\begin{split}
&z\frac{ H^{(1)'}_\nu(z)}{H_\nu^{(1)}(z)}=\ri z-\frac{1}{2}+\sum\limits_{j=1}^{N_{\nu}}\frac{h_{\nu,j}}{z-h_{\nu,j}}\\
&\quad  -\frac{1}{\pi \ri}\int_0^{\infty}\frac{\pi\cos(\nu\pi)}{\cos^2(\nu\pi)K_{\nu}^2(r) +(\pi
I_{\nu}(r)+\sin(\nu\pi)K_{\nu}(r))^2} \frac 1 {\ri r+z} dr,
\end{split}
\end{equation}
for any $\nu\not=n+1/2,$ where
$h_{\nu,1},h_{\nu,2},\cdots,h_{\nu,N_{\nu}}$ are zeros of
$H_\nu^{(1)}(z),$ which number $N_\nu.$  Interestingly,
\eqref{sigmnt2d} can be derived from  \eqref{znuschform}, which is
justified in Appendix  {\rm \ref{justappB}}. \qed
\end{remark}

\begin{remark} \label{ksweiomeg}
 Thanks to  \eqref{nnectoon}, we obtain  from  Theorem {\rm  \ref{theorem1}} the
 expression of $\omega_\nu(t):$
\begin{itemize}
\item for $d=2,$
\begin{equation}\label{sigmnt2d2s}
\omega_n(t)=-\frac c {2b}+\frac c b \bigg\{ \sum_{j=1}^{M_n}\big(
e^{ct z_{j}^n/b}-1\big) +(-1)^n \int_0^{\infty}\frac{ 1- e^{-ctr/b}}
{r\{K^2_{n}(r)+\pi^2I^2_{n}(r)\}}dr\bigg\},
\end{equation}
\item  for $d=3,$
\begin{equation}\label{sigmnt2d2sw}
\omega_\nu(t)=-\frac c {b}+\frac c b
\sum_{j=1}^{M_{\nu}}\big(e^{ctz_{j}^{\nu}/b}-1\big),\quad
\nu=n+ 1/2.  \qed
\end{equation}
\end{itemize}
 \end{remark}

\subsection{Computation of the  improper integral in \eqref{sigmnt2d}}\label{impComLab}
The computation of the two-dimensional NRBK requires to evaluate the
improper integral involving the kernel function:
\begin{equation}\label{KnInrela}
W_n(r):=\frac 1  {K_n^2(r)+\pi^2 I_n^2(r)}:=\frac 1 {G_n(r)},\quad n\ge 0,\;\; r>0,
\end{equation}

whose important properties are characterized below.
\begin{lemma}\label{wnkernl}  For any $n\ge 0$ and any real $r>0,$   we have %
\begin{itemize}
\item[(i)]      $G_n(r)$  is a convex function of $r,$  and $W_n(r)$ attains its maximum at a unique  point.

\item[(ii)]   For large $n,$  we have the uniform asymptotic estimate:
 \begin{equation}\label{asfun}
W_n(n\kappa) \sim \frac{n\sqrt{1+\kappa^2}} \pi  {\rm sech}(2n\Theta):=\widetilde W_n(n\kappa),
\end{equation}
for $\kappa>0,$ where
\begin{equation}\label{Thetakappa}
 \Theta=\Theta(\kappa):=\sqrt{1+\kappa^2}+\ln \frac \kappa {1+\sqrt {1+\kappa^2}}.
\end{equation}
 Approximately,  the maximum value of $W_n(r)$ attains at $r=na$ with $a\approx 0.66274$ being the root  of $\Theta,$ and the maximum value is approximately $n\sqrt{1+a^2}/\pi\approx 0.38187n.$
 \end{itemize}
\end{lemma}
\begin{proof} (i).  We find from Page 374 of \cite{Abr.S84} that   for a given $n,$ $K_n(r), I_n(r)>0, $  and    $K_n(r)$ (resp. $I_n(r)$) is monotonically  descending (resp. ascending) with respect to $r.$  From the series representation (see Formula 9.6.10  of \cite{Abr.S84}):
\[
I_n(r)=\frac{r^n}{ 2^n}\sum_{k=0}^\infty \frac{r^{2k}}{2^{2k} k! (n+k)!},\quad n\ge 0,
\]
we conclude that $I_n''(r)>0.$ Moreover, since $K_n(r)$ satisfies (see Formula 9.6.1  of \cite{Abr.S84})
\[
r^2 K_n''(r)+r K_n'(r)-(r^2+n^2) K_n(r)=0,
\]
we have $K_n''(r)>0.$ Therefore,  a direct calculation shows that  $G_n''(r)>0,$
so $G_n(r)$ is convex.

One verifies  that $G_n(0+)=G_n(+\infty)=+\infty$ for all $n,$ which
follows from \eqref{asympt}, and the asymptotic
properties (see  \cite{Abr.S84} again):
\begin{equation}\label{asymptab}
\begin{split}
K_\nu (r)\sim \begin{cases}
-\ln r,\;\; &{\rm if}\;\; \nu=0,\\[7pt]
\dfrac{\Gamma(\nu)} {2} \Big(\dfrac r 2\Big)^{-\nu},\;\; &{\rm
if}\;\; \nu>0,
\end{cases}
\end{split}
\qquad
I_\nu (r)\sim \frac 1 {\Gamma(\nu+1)} \Big(\frac r 2\Big)^{\nu}, \;\; \nu\ge 0,
\end{equation}
for $0<r\ll 1,$ 
and  \begin{equation}\label{asymptabc}
I_\nu(r)\sim \sqrt{\frac 1 {2\pi r}} e^r,\quad r\gg1,\;\; \nu\ge 0.
\end{equation}
Since $G_n(r)$ is convex,     $G_n(r)$ attains its  minimum  at a unique   point $r_0.$
Thanks to    $G_n'(r)=-W_n'(r)/W_n^2(r),$ $W_n(r)$ has a unique maximum at the same point $r_0.$

(ii).  Recall that for large $n$ (see
Formulas (9.7.7)-(9.7.11) of \cite{Abr.S84}):
\begin{equation}\label{constestt}
\begin{split}
&K_n(n\kappa)\sim \sqrt{\frac \pi {2n}}\frac{e^{-n\Theta}}{(1+\kappa^2)^{1/4}}, \qquad\qquad   I_n(n\kappa)\sim \frac 1 {\sqrt{2\pi n}} \frac{e^{n\Theta}}{(1+\kappa^2)^{1/4}},\\
& K_n'(n\kappa)\sim -\sqrt{\frac \pi {2n}}
\frac{(1+\kappa^2)^{1/4}} \kappa
{e^{-n\Theta}} ,\quad  I_n'(n\kappa)\sim \frac 1 {\sqrt{2\pi n}} \frac{(1+\kappa^2)^{1/4}} \kappa {e^{n\Theta}},
\end{split}
\end{equation}
which, together with \eqref{KnInrela},  leads to the asymptotic estimate \eqref{asfun}.
Thus,  the maximum value of $W_n(r)$  approximately attains at the unique root  of  $\Theta(\kappa),$
 which turns out to be $a\approx 0.66274$ as in Figure \ref{Figzerodistr}, and the maximum value is about $n\sqrt{1+a^2}/\pi\approx 0.38187n.$
 \end{proof}

\begin{figure}[!ht]
\includegraphics[width=0.98\textwidth,height=0.35\textwidth]{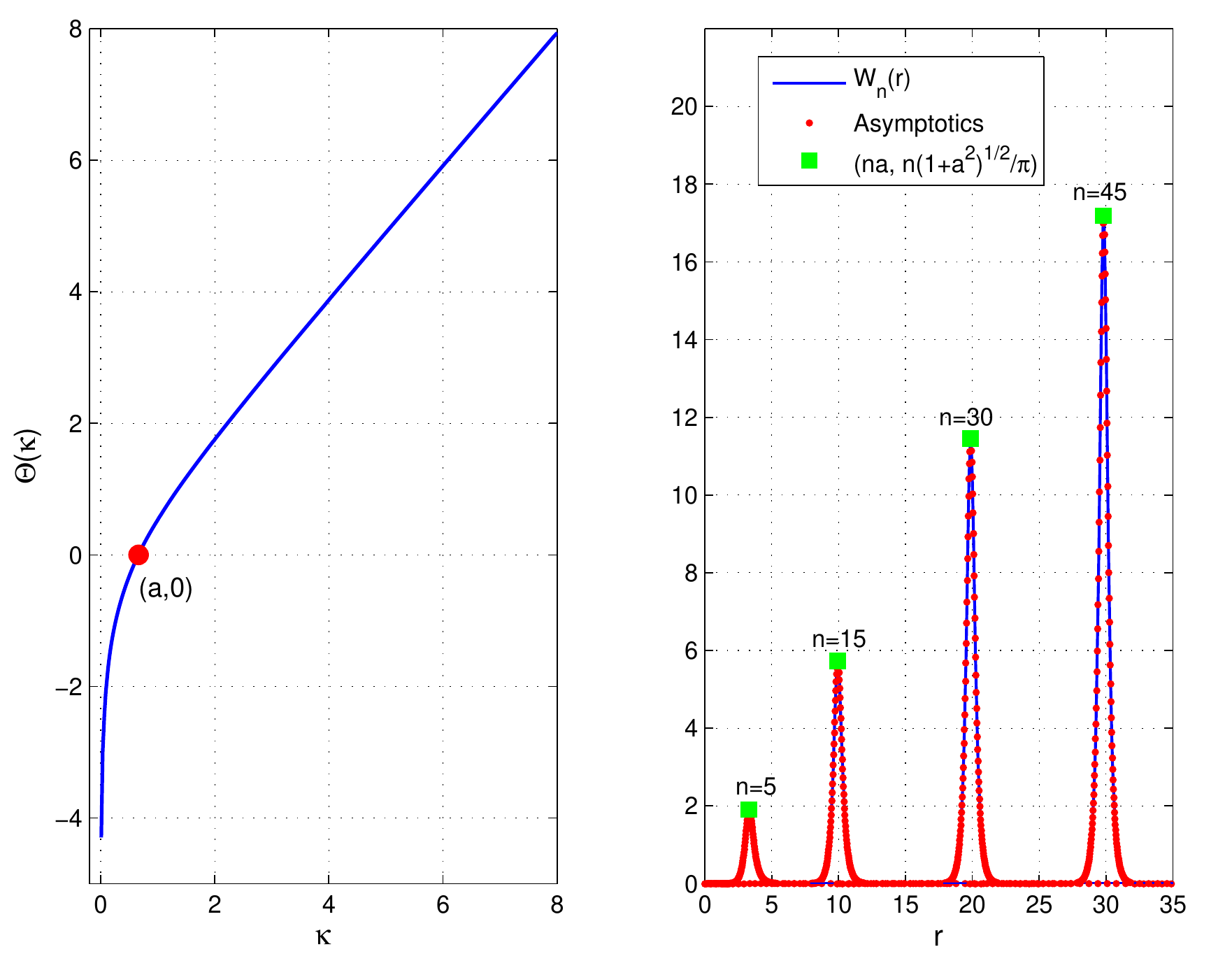}
\caption{\small Left: graph of $\Theta(\kappa), \kappa\in (0,8]$ defined in
(\ref{Thetakappa}). Right: $W_n$ against the asymptotic estimate $\widetilde W_n$ (``{\color{red}$\bullet$}") for samples of $n=5,15, 30, 45.$}\label{FigWnrfig}
\end{figure}

We depict in Figure \ref{FigWnrfig} (left) the graph of $\Theta(\kappa)$, and  highlight the zero point $(a,0).$ Observe that $\Theta(\kappa)$ grows like $\kappa,$ when $\kappa>a.$
We also plot in Figure \ref{FigWnrfig} (right) several sample graphs of  $W_n$ (solid lines) and the asymptotic estimate
$\widetilde W_n$ (``{\color{red}$\bullet$}") for $n=5,15,30,45,$ and particularly mark the asymptotic point  $(na, n\sqrt{1+a^2}/\pi)$
of the maximum of $W_n(r)$ obtained in Lemma \ref{wnkernl} (ii). Observe that  even for small
 $n,$  the asymptotic estimate provides a very accurate  approximation  of $W_n.$
These properties  greatly
facilitate the computation of the improper integral. Indeed, we can
truncate  $(0, \infty)$ to a narrow interval (of length about $20$ for $n\ge 5$ in our numerical computations in Section \ref{sect5})
around the point $r=na.$

\subsection{Rapid  evaluation of the temporal  convolution}\label{Sect:rapideva}  Remarkably,  the presence of the time variable $t$
in the exponentials in \eqref{sigmnt2d} and \eqref{sigmnt3d}
allows us to eliminate the burden of history dependence of the temporal convolution easily.
 More precisely,  given a function  $g(t),$ we define
 \begin{equation*}\label{recurdeltabc}
f(t;r):=e^{-ctr/b}\ast g(t)=\int_0^t e^{-c(t-\tau)r/b} g(\tau) d\tau.
\end{equation*}
One verifies readily that
\begin{equation}\label{recurdeltab}
\begin{split}
 f(t+\Delta t; r)=e^{-c\Delta t r/b} f(t; r)+\int_{t}^{t+\Delta t}e^{-c(t+\Delta t-\tau)r/b} g(\tau) d\tau,
\end{split}
\end{equation}
so $f(t;r)$ can  march in  $t$ with step size $\Delta t$  recursively.
This enables us to compute the time convolution rapidly. For example, in the 2-D case,
\begin{equation}\label{recurdelta}
\begin{split}
[\sigma_n\ast g](t)&=\int_0^t \sigma_n(t-\tau)g(\tau) d\tau=\frac c {b^2}
\sum_{j=1}^{M_n}z_j^n \int_0^t e^{c(t-\tau)z_j^n/b}g(\tau)d\tau \\
&\quad + \frac{(-1)^n c}{b^2} \int_0^\infty \frac 1 {K_n^2(r)+\pi^2 I_n^2(r)}\Big[\int_0^t e^{-c(t-\tau)r/b} g(\tau) d\tau\Big]dr\\
&=\frac c {b^2} \sum_{j=1}^{M_n}z_j^n f(t;-z_j^n)+\frac{(-1)^n c}{b^2} \int_0^{+\infty} {f(t; r)} W_n(r)  dr.
\end{split}
\end{equation}
Thanks to \eqref{recurdeltab}, $[\sigma_n\ast g](t+\Delta t)$ can be computed  recursively from the previous step and the history dependence is then narrowed down to $[t,t+\Delta t].$
We also refer to Subsection \ref{Sect:Newmark} for more detailed discussions.

\vskip 10pt
\section{A priori estimates}
\setcounter{equation}{0}

In this section, we analyze the well-posedness of the
 truncated problem \eqref{eq1.1ab}-\eqref{eq1.1ce}, and
 provide  {\em a priori} estimates for  its solution.

We first recall the Plancherel or Parseval results for the Laplace transform.
 \begin{lemma}\label{ParIden}  Let $s=s_1+\ri s_2$ with  $s_1, s_2\in {\mathbb R}.$ If $f,g$ are Laplace transformable,  then
 \begin{equation}\label{Pariddd}
 \frac 1 {2\pi}\int_{-\infty}^{+\infty} {\mathcal L}[f](s)\, {\mathcal L}[\bar g](s)\, ds_2=\int_0^{+\infty} e^{-2s_1 t} f(t) \bar g(t) dt,\quad  \forall\, s_1>\gamma,
 \end{equation}
  where   $\gamma$ is  the absissa of convergence for both $f$ and $g$,   and   $\bar g$ is the complex conjugate of $g.$
 \end{lemma}
  \begin{proof}  This identity  can be proved by following the same lines as for  (2.46) in \cite{CohenLap2007}
   (or \cite{DaviesBrian2002}).
  \end{proof}

For notational convenience, we introduce the modified spherical Bessel function (cf. \cite{watson}):
\begin{equation}\label{akasympc}
k_{n}(z)=\sqrt{\frac 2 {\pi z}} K_{n+1/2}(z)\;\; \Rightarrow \;\;  \frac{k_n'(z)}{k_n(z)}=-\frac 1 {2z}+\frac{K_{n+1/2}'(z)}{K_{n+1/2}(z)},
\end{equation}
and by \eqref{kernalw},
\begin{equation}\label{timeasddtcd}
\omega_{n+1/2}(t)=\mathcal{L}^{-1}\left[1+\frac{k_n'
(sb/c)}{k_n(sb/c)}\right](t),\quad n\ge 0.
\end{equation}
The following properties are also indispensable  for the analysis.
\begin{lemma}\label{kenpos} Let $s=s_1+\ri s_2$ with  $s_1, s_2\in {\mathbb R}.$ Then for any $s_1>0,$
\begin{equation}\label{knposi}
{\rm (i)}\; {\rm Re}\Big(s\frac{Z_n'(sb/c)} {Z_n(sb/c)}\Big)\le 0, \;\;  {\rm Re}\Big(\frac{Z_n'(sb/c)} {Z_n(sb/c)}\Big)\le 0; \;\;   {\rm (ii)}\; {\rm Im}\Big(s\frac{Z_n'(sb/c)} {Z_n(sb/c)}\Big)\le 0,\;\; \forall s_2\ge 0,
\end{equation}
where $Z_n(z)=K_n(z)$ or $k_n(z).$
\end{lemma}
\begin{proof}
The  results with $Z_n=K_n$ were proved in Chen \cite{Chen.ZM2009}.
We next prove \eqref{knposi}
  with $Z_n=k_n$ by using a similar argument. By
applying Laplace transform to
\eqref{eq1.1a}-\eqref{eq1.1c}, exterior to $D=B$ with
$F=U_0=U_1\equiv 0$ and $G=U|_{r=b},$ and denoting $u={\mathcal L}[U],$ we obtain
\begin{equation}\label{eq1.2}
\begin{split}
& -c^2\Delta u+s^2 u=0,\quad {\rm in}\;\; \Omega_{\rm ext}={\mathbb R}^d\setminus \bar B,\;\;  s\in {\mathbb C},\;\;  {\rm Re}(s)>0;\\
&  u|_{r=b}=\psi=\mathcal{L}[U|_{r=b}]; \quad  c {\partial_r u}+ s
u=o\big(r^{(1-d)/2}\big),\;\;  d=2,3,
 \end{split}
\end{equation}
which admits the series solution
\begin{equation}\label{sersol3d}
u(r,\theta,\phi,s)=\sum_{n=0}^\infty \frac{k_{n}(sr/c)}{k_{n}(sb/c)}
\sum_{|m|=0}^n \hat \psi_{nm}(s) Y_{n}^m(\theta,\phi),
\end{equation}
where $\{\hat \psi_{nm}\}$ are the spherical harmonic expansion
coefficients of $\psi$.  Multiplying the first equation of
\eqref{eq1.2}   by $\bar u$ and integrating over
$\Omega_{b,\rho}:=B_\rho\setminus \bar B,$ where $B_\rho$ is a ball
of radius $\rho>b,$ the imaginary part of the resulting equation
reads
\begin{equation}\label{interimaginary}
2s_1s_2\int_{\Omega_{b,\rho}} |u|^2 d\bs x
-c^2\, {\rm Im} \int_{\partial\Omega_{b,\rho}}\frac{\partial u}{\partial {\bs n}}\bar u\, d\gamma=0,
\end{equation}
where $\bs n$ is the unit outer normal of  $\partial\Omega_{b,\rho}$.

Since $s_1>0,$  multiplying \eqref{interimaginary} by $s_2$ yields
\begin{equation}\label{IntermediaF}
s_2\,{\rm Im} \int_{\{r=b\}} \frac{\partial u}{\partial r}\bar u\, d\gamma
 \le s_2\,{\rm Im} \int_{\{r=\rho\}} \frac{\partial u}{\partial r}\bar u\, d\gamma.
\end{equation}
It is clear that $u=k_n(sr/c)/k_n(sb/c) \hat
\psi_{nm}(s)Y_n^m(\theta,\phi)$ is a solution of  \eqref{eq1.2}, so
using the orthogonality of $\{Y_n^m\},$ we obtain from
\eqref{IntermediaF} that
\begin{equation}\label{Kbnns}
\frac {b} c {\rm Im}\Big(s_2 s \frac{k_n'(sb/c)}{k_n(sb/c)}\Big)|\hat \psi_{nm}|^2 \le \frac {\rho}
c {\rm Im}\Big(s_2 s \frac{k_n'(s\rho/c)}{k_n(sb/c)}\Big)|\hat \psi_{nm}|^2.
\end{equation}
We find from  \eqref{asympt} and \eqref{akasympc} that $k_n'(s\rho/c)$
decays exponentially if $s_1>0,$ so the right hand side of
\eqref{IntermediaF} tends to zero as $\rho\to +\infty.$ Thus,
letting $\rho\to +\infty$ in \eqref{Kbnns} leads to
  \begin{equation}\label{Kbnns2}
  {\rm Im}\Big(s_2 s \frac{k_n'(sb/c)}{k_n(sb/c)}\Big)\le 0.
\end{equation}
If $s_2\ge 0,$  we obtain (ii) in \eqref{knposi}. 
Next, we prove the first inequality of (i) in \eqref{knposi}. Recall
the formula (see Lemma 2.3 of \cite{Chen.ZM2009}):
  \begin{equation*}\label{Kbnns2wwg}
|K_{n+1/2}(sr)|^2=\frac 12 \int_0^{+\infty} e^{-\frac {|s|^2r^2}{2r}-\frac{s^2+\bar s^2}{2|s^2|}} K_{n+1/2}(\tau) \frac {d\tau} \tau,\quad s_1>0,
 \end{equation*}
 which implies that $|K_{n+1/2}(sr)|^2$ is monotonically descending with respect to $r.$ The property $\frac d{dr} |K_{n+1/2}(sr)|^2\le 0,$
  together with \eqref{akasympc}, implies ${\rm Re}\Big(s\frac{k_n'(sb/c)} {k_n(sb/c)}\Big)\le 0$.
Denoting $s \frac{k_n'(sb/c)}{k_n(sb/c)}=\gamma_1+\ri \gamma_2$ with
$\gamma_1,\gamma_2\in {\mathbb R},$  we know from \eqref{Kbnns2}
that $\gamma_1\le 0$ and $s_2\gamma_2\le 0.$ Therefore,
\[
{\rm Re}\Big(\frac{k_n'(sb/c)} {k_n(sb/c)}\Big)={\rm Re}\Big(\frac {\gamma_1+\ri\gamma_2}{s}\Big)
=\frac 1 {|s|^2}(s_1\gamma_1+s_2\gamma_2)\le 0.
\]
This ends the proof.
\end{proof}

With the above preparations,  we can derive the following  important property.
\begin{theorem}\label{fgprops}  For any  $v\in L^2(0,T),$ we have
\begin{equation}\label{fgss}
\int_0^T [\omega_\nu\ast v](t) \bar v(t) dt   \le  \int_0^T|v(t)|^2
dt,\quad \forall\; T>0,\;\; n\ge 0,
\end{equation}
where $\omega_\nu$ is the  NRBK given by  \eqref{kernalw} {\rm(}or
 \eqref{timeasddtcd} for $d=3${\rm)}.
\end{theorem}
\begin{proof}  Let $\tilde v=v \, {\bf 1}_{_{[0,T]}},$ where  ${\bf 1}_{_{[0,T]}}$ is the characteristic function of $[0,T].$
Then we obtain from  \eqref{Pariddd}  that for $d=3,$
\begin{equation*}\label{Lulv}
\begin{split}
 \int_0^T e^{-2s_1 t}  & [\omega_{n+1/2}\ast v](t) \bar v(t) dt  = \int_0^{+\infty} e^{-2s_1 t}  [\omega_{n+1/2}\ast \tilde v](t)  \bar {\tilde v}(t) dt\\
& =\frac 1 {2\pi} \int_{-\infty}^{+\infty} \Big[ \frac{k_n'(sb/c)}
{k_n(sb/c)}+1\Big] \big|{\mathcal L}[\tilde v](s)\big|^2 d s_2
\\&=
\frac 1 {2\pi} \int_{-\infty}^{+\infty}
\frac{k_n'(sb/c)} {k_n(sb/c)} \big|{\mathcal L}[\tilde v](s)\big|^2 d s_2
 + \frac 1 {2\pi} \int_{-\infty}^{+\infty}
 \big|{\mathcal L}[\tilde v](s)\big|^2 d s_2.
\end{split}
\end{equation*}
It is clear that  by \eqref{Pariddd} with $f=g=\tilde v,$
\[
\frac 1 {2\pi} \int_{-\infty}^{+\infty}
 \big|{\mathcal L}[\tilde v](s)\big|^2 d s_2=\int_0^{+\infty} e^{-2s_1t} |\tilde v(t)|^2 dt=\int_0^T e^{-2s_1t}|v(t)|^2 dt.
\]
Using \eqref{asympta} and the properties  $\overline {K_\nu(z)}=K_\nu(\bar z)$ and ${\mathcal L}[\tilde v](\bar s)=\overline{{\mathcal L}[\tilde v](s)},$ we find
\begin{equation}\label{arrf}
\begin{split}
\frac 1 {2\pi} \int_{-\infty}^{+\infty}
\frac{k_n'(sb/c)} {k_n(sb/c)}& \big|{\mathcal L}[\tilde v](s)\big|^2 d s_2=\frac 1{\pi}
\int_0^{+\infty}{\rm Re}\Big(\frac{k_n'(sb/c)} {k_n(sb/c)} \Big)  \big|{\mathcal L}[\tilde v](s)\big|^2 d s_2\\
&=\frac 1{\pi}
\int_0^{+\infty}{\rm Re}\Big(\frac{k_n'(sb/c)} {k_n(sb/c)} \Big)|s|^2  \Big|{\mathcal L}
\Big[\int_0^t\tilde v(\tau)d\tau\Big](s)\Big|^2 d s_2.
\end{split}
\end{equation}
Thus, we conclude from  Lemma \ref{kenpos} and the above identities that for $s_1>0,$
\begin{equation}\label{otrest}
 \int_0^T e^{-2s_1 t}   [\omega_{n+1/2}\ast v](t)  \bar v(t) dt\le \int_0^T e^{-2s_1t}|v(t)|^2 dt.
\end{equation}

Notice that  the asymptotic formulas in  \eqref{asymptab} are also
valid for complex $r$ (see Formula 9.6.9 of \cite{Abr.S84}), which,
together with   Lemma \ref{lemma2}, implies that
$k_n'(sb/c)/k_n(sb/c)$ is analytic for all ${\rm Re}(s)\ge 0$ but
$|s|\not =0,$ and $\lim_{s_1\to 0^+} |s|^2k_n'(sb/c)/k_n(sb/c)$
exists for all $s_2\ge 0.$ Hence, the integral in \eqref{arrf} is
finite. By letting $s_1\to 0^+$  in \eqref{otrest} leads to the
desired result for $d=3.$

The result \eqref{fgss} with $d=2$ can be proved in a similar fashion.
\end{proof}

\begin{corollary}\label{sigma} Suppose that $ v'\in L^2(0,T)$ with $v(0)=0.$ Then we have
\begin{equation}\label{fgsswang}
\int_0^T [\sigma_\nu \ast v](t)\bar   v'(t) dt
 \le  \frac 1 c \int_0^T|v'(t)|^2 dt+\frac {d-1} {4b}|v(T)|^2 ,\quad \forall\; T>0,
\end{equation}
for $\nu=n,n+1/2,$ where $\sigma_\nu(t)$ is the NRBK defined in
\eqref{kernalf}.
\end{corollary}
\begin{proof} By \eqref{nnectoon}, we have
$\omega_\nu'(t)=c\sigma_\nu(t), \omega_\nu(0)=-\frac {(d-1)c} {2b}.$
Thus, we obtain from integration by parts and the fact $v(0)=0$ that
\[
[\omega_\nu\ast v'](t)=-\frac {(d-1)c}{2b} v(t)+c[\sigma_\nu \ast
v](t).
\]
By Theorem \ref{fgprops} with $v'$ in place of $v,$
\begin{equation*}
\begin{split}
\int_0^T & [\omega_\nu\ast v'](t) \bar v'(t)dt = -\frac {(d-1)c} {2b} \int_0^T v(t)\bar v'(t) dt+c\int_0^T[\sigma_\nu\ast v](t)\bar v'(t) dt\\
&=-\frac {(d-1)c}  {4b} {|v(T)|^2}+ c\int_0^T[\sigma_\nu\ast v](t)
\bar v'(t) dt\le \int_0^T|v'(t)|^2 dt.
\end{split}
\end{equation*}
This gives \eqref{fgsswang}.
\end{proof}

Now, we are ready to analyze the stability of the solution of the truncated problem  \eqref{eq1.1ab}-\eqref{eq1.1ce}.  To this end, we assume that the scatter $D$ is a simply connected domain with Lipschitz boundary $\Gamma_D$, and the Dirichlet data $G=0$ on $\Gamma_D.$  Denote
$X:=\big\{U\in H^1(\Omega) :  U|_{\Gamma_D}=0\big\},$ and
let $(\cdot,\cdot)_{L^2(\Omega)}$ and $\|\cdot\|_{L^2(\Omega)}$ be  the inner product and norm of
$L^2(\Omega),$ respectively.
\begin{theorem}\label{stabilcont} Let $U (\in X)$ be the solution of \eqref{eq1.1ab}-\eqref{eq1.1ce} with $G=0.$
If $U_0\in H^1(\Omega), U_1\in L^2(\Omega)$ and $F\in
L^1(0,T;L^2(\Omega))$ for any $T>0,$ then we have $\nabla U\in
L^\infty(0,T;(H^1(\Omega))^d),$ $ \partial_t U\in L^\infty(0,T;
L^2(\Omega)),$ and there holds
\begin{equation}\label{stabestt}
\begin{split}
\|\partial_t U\|_{L^\infty(0,T;L^2(\Omega))} & +c\|\nabla
U\|_{L^\infty(0,T;L^2(\Omega))}\\
&\le C\big( \|U_1\|_{L^2(\Omega)} +{c}\|\nabla
U_0\|_{L^2(\Omega)}+\|F\|_{L^1(0,T;L^2(\Omega))}\big),
\end{split}
\end{equation}
where $C$ is a positive constant independent of any functions and
$c,b.$ Moreover, if the source term $F\equiv 0$, we have the
conservation of the energy: $E'(t)=0$ for all $t\ge 0,$ where
\begin{equation}\label{energyeqn}
 E(t)=\int_\Omega\big(|\partial_t U|^2+c^2|\nabla
U|^2\big)d\bs x - 2c^2\int_0^t \int_{\Gamma_b}T_{d}(U)\partial_\tau
\overline{U} d\gamma d\tau.
\end{equation}
\end{theorem}
\begin{proof}
Multiplying \eqref{eq1.1ab} by $2\partial_t \overline{U}$ and
integrating over $\Omega,$ we derive from the Green's formula that
for any $t>0,$
\begin{eqnarray}\label{pf133}
\frac{d}{dt}\Big(\|\partial_t U\|^2_{L^2(\Omega)}+{c^2} \|\nabla
U\|^2_{L^2(\Omega)}\Big)- 2c^2 \int_{\Gamma_b}T_{d}(U)\partial_t
\overline{U} d\gamma=2(F,\partial_t U)_{L^2(\Omega)}.
\end{eqnarray}
Integrating  the above equation over $(0,t),$ we find that for any $t>0,$
\begin{eqnarray}\label{pf13}
\begin{split}
\|\partial_t U\|^2_{L^2(\Omega)}& +{c^2} \|\nabla
U\|^2_{L^2(\Omega)}- 2c^2\int_0^t
\int_{\Gamma_b}T_{d}(U)\partial_\tau\overline{U} d\gamma d\tau  \\
&=2\int_0^t (F,\partial_\tau  U)_{L^2(\Omega)}d\tau +
\|U_1\|^2_{L^2(\Omega)} +{c^2}\|\nabla U_0\|^2_{L^2(\Omega)}\\
&\le 2\|\partial_t U\|_{L^\infty(0,T; L^2(\Omega))}\|F\|_{L^1(0,T;L^2(\Omega))}+
\|U_1\|^2_{L^2(\Omega)} +{c^2}\|\nabla U_0\|^2_{L^2(\Omega)}.
\end{split}
\end{eqnarray}
We next show that for any $t>0,$
\begin{equation}\label{Funcat}
\int_0^t  \int_{\Gamma_b}T_{d}(U)\partial_\tau \overline{U} d\gamma
d\tau \le 0.
\end{equation}
For  $d=3,$ it follows from  \eqref{GdUaddnew},
 Theorem \ref{fgprops} and the orthogonality of $\{Y_n^m\}$ that
\begin{equation*}\label{Funcatgg}
\begin{split}
\int_0^t \int_{\Gamma_b}& T_{d}(U)\partial_\tau \overline{U} d\gamma
d\tau=-\frac 1 c
\int_0^t\|\partial_\tau U\|^2_{L^2(\Gamma_b)} d\tau\\
&\quad +\frac 1 c \sum_{n=0}^\infty \sum_{|m|=0}^n \int_0^t \big[\omega_{n+1/2}\ast \partial_\tau \widehat U_{nm}(b,\tau)\big] {\partial_\tau \overline{\widehat U}_{nm}(b,\tau)} d\tau \\
&\overset{(\ref{fgss})}\le -\frac 1 c \int_0^t\|\partial_\tau
U\|^2_{L^2(\Gamma_b)} d\tau + \frac 1 c \sum_{n=0}^\infty
\sum_{|m|=0}^n\int_0^t \big|\partial_\tau \widehat
U_{nm}(b,\tau)\big|^2d\tau=0.
\end{split}
\end{equation*}
This verifies \eqref{Funcat} with $d=3.$ Similarly, one can justify  \eqref{Funcat} with $d=2.$
Consequently, the estimate \eqref{stabestt} follows from
\eqref{pf13}, \eqref{Funcat}  and the Cauchy-Schwarz inequality.

Letting $F\equiv 0$ in \eqref{pf13} leads to  the conservation of energy.
\end{proof}

\begin{remark}\label{qaremark} Note that
{\rm (i)} if  $T_d(U)$ is given by \eqref{GdUadd}, we can use  Corollary
{\rm \ref{sigma}} to verify \eqref{Funcat}, and {\rm (ii)}
  a similar estimate for $d=2$ was derived by \cite{Chen.ZM2009} with a slightly different argument. \qed
\end{remark}

\vskip 10pt

\section{ Spectral-Galerkin approximation and  Newmark's time integration}\label{sect4}

This section is devoted to  numerical approximation of the truncated
problem \eqref{eq1.1ab}-\eqref{eq1.1ce} with a focus on the
treatment for the  NRBC \eqref{eq1.1ce}. Note that if the
scatter $D$ is a disk/ball (i.e., $\Omega$ in
\eqref{eq1.1ab}-\eqref{eq1.1ce} is an annulus/spherical shell), the
nonlocality of  the NRBC in space becomes local in the space of
Fourier/spherical harmonic coefficients
 in polar/spherical coordinates.  Correspondingly, the truncated problem can be reduced to a sequence of one-dimensional
 problems with mixed boundary conditions at the outer artificial boundary.
 In order to  deal with a general scatter, one may resort to
  the transformed field expansion method (TFE) (cf.  \cite{Nic.R03}), which has been successfully
  applied to  time-harmonic Helmholtz equations (cf. \cite{Nic.S06,Fang.DShen07,Nich.Shen2009}). The use of this method
  allows  us to  solve a sequence of  truncated problems \eqref{eq1.1ab}-\eqref{eq1.1ce} in  regular domains.
  Accordingly, it is essential to construct a fast and accurate  solver for
 \eqref{eq1.1ab}-\eqref{eq1.1ce} in an annulus or a spherical shell, which will be the concern of this section.
We shall report the  combination of the solver with the TFE method in  a future work as the implementation is rather involved.

\subsection{Notion of TFE}  We outline the TFE method for \eqref{eq1.1ab}-\eqref{eq1.1ce} with $d=2$
and
 \[
 \Gamma_D=\big\{r=b_0+\eta(\phi) : 0\le \phi<2\pi \big\},\quad  b>b_0+\max_{\phi\in [0,2\pi]}|\eta(\phi)|,
 \]
 where $r(\phi)=b_0+\eta (\phi)$ is the parametric equation of the boundary of the scatter $D$.
 \begin{itemize}
\item  Make a change of variables
 \[
 r'=\frac{(b-b_0)r-b\eta(\phi)}{(b-b_0)-\eta(\phi)},\quad \phi'=\phi,
 \]
 which maps $\Omega$ to the   annulus $\Omega_0=\big\{(r',\phi') : b_0<r'<b,\, 0\le \phi' <2\pi \big\}.$ To simplify the notation, we still use $U,F, r, \phi$ etc. to denote the transformed functions or variables.
Then the problem  \eqref{eq1.1ab}-\eqref{eq1.1ce} becomes
 \begin{equation*}
 \begin{split}
&\partial_t^2 U=c^2\Delta U+F+J(\eta,U),\;\; {\rm in}\;\; \Omega_0, \;\; t>0;
\\& U=U_0,\;\; \partial_t U=U_1,\;\; {\rm in}\;\;\Omega_0, \;\; t=0; \\
& U|_{r=b_0}=G,\;\;\;  t>0; \;\; \big(\partial_r U-T_d(U)\big)|_{r=b}=L(\eta, U)|_{r=b},\quad  t>0,
  \end{split}
\end{equation*}
where $J(\eta,U)$ and $L(\eta,U)$ contain differential operators with nonconstant coefficients.

 \item To solve the transformed problem efficiently, we adopt the boundary perturbation technique by viewing the obstacle as a perturbation of the disk. More precisely, we write $\eta=\varepsilon \zeta$ and expand  the solution $U$ as
$U(r,\phi, t; \varepsilon) =\sum_{l=0}^\infty U_l(r,\phi,t) \varepsilon^l,$
 and likewise for the data. Formally, we obtain the sequence of equations after collecting the terms of $\varepsilon^l$ (see \cite{Nic.S06}):
 \begin{equation*}
 \begin{split}
&\partial_t^2 U_l=c^2\Delta U_l+F_l+\tilde J(\eta,U_{l-4},\cdots, U_{l-1}),\quad {\rm in}\;\; \Omega_0, \;\; t>0; \\
&U_l=U_{0,l},\quad \partial_t U_l=U_{1,l}, \quad {\rm in}\;\;\Omega_0, \;\; t=0; \\
& U_l|_{r=b_0}=G_l,\;\;\;  t>0; \quad  \big(\partial_r
U_l-T_d(U_l)\big)|_{r=b}=\tilde L(\eta, U_{l-1})|_{r=b},\quad  t>0.
  \end{split}
\end{equation*}
\item Solve the above equation for $l=0,1,\cdots$, and sum up the series by using a Pad$\acute{e}$ approximation.
\end{itemize}

\subsection{Prototype equation and dimension reduction}
 Under this notion, the whole algorithm boils down to solving a sequence of   prototype equations. More precisely, we consider
\begin{equation}\label{sprciprob}
\begin{split}
& \partial_t^2 U=c^2\Delta U+F,\quad  {\rm in}\;\;  \Omega_0=\big\{\bs x\in {\mathbb R}^d : b_0<|\bs x|<b \big\},\;\; t>0; \\
& U=U_0,\;\; \partial_t U=U_1, \;\;  {\rm in}\;\;\Omega_0, \;\; t=0;\quad
 U|_{r=b_0}=0, \;\;  \big(\partial_r U-T_d(U)\big)\big|_{r=b}=0, \;\; t>0,
\end{split}
\end{equation}
where $T_d(U)$ is the  DtN map as before.  We expand the solution
and given data in Fourier series/spherical harmonic series. Then the
problem \eqref{sprciprob}, after a polar (in 2-D) and spherical (in
3-D) transform, reduces to a sequence of one-dimensional problems
(for brevity, we use $u$ to denote the Fourier/spherical harmonic
expansion coefficients  of $U,$ and likewise, we use $u_0, u_1$ and
$f$ to denote the expansion coefficients of $U_0, U_1$ and $F,$
respectively):
\begin{equation}\label{sprciprob2}
\begin{split}
& \frac{\partial^2 u}{\partial t^2}-\frac {c^2} {r^{d-1}}
\frac{\partial}{\partial r} \Big(r^{d-1}
\frac {\partial u}{\partial r}\Big)+ c^2\beta_n \frac u{r^2}=f,\quad   b_0<r<b,\;\; t>0; \\
& u|_{t=0}=u_0,\quad \frac{\partial u}{\partial t}\Big|_{t=0}=u_1, \quad  b_0< r< b; \quad  u|_{r=b_0}=0,\;\; t>0; \\
  & \Big(\frac 1 c \frac {\partial u}{\partial t} + \frac {\partial u}{\partial r}+\frac {d-1}{2r}u\Big)\Big|_{r=b}=\int_0^t  \sigma_\nu(t-\tau) u(b,\tau) d\tau, \;\; t>0,
\end{split}
\end{equation}
where $\beta_n=n^2, n(n+1)$ and $\nu=n, n+1/2$ for $d=2,3,$ respectively.

Since  $\sigma_\nu$ is real, the real and imaginary parts of $u$ and the given data can
be decoupled. In what follows, we assume they are real.

\subsection{Spectral-Galerkin approximation in space}  We apply the Legendre spectral Galerkin
method to approximate
\eqref{sprciprob2} in space.  For convenience of implementation, we transform the interval $(b_0,b)$ to the
reference interval $I=(-1,1)$ by  $r=\frac {b-b_0} 2 x+\frac {b+b_0} 2$ with $x\in \bar I,$ and denote
 the transformed functions by $v(x,t)=u(r,t), h(x,t)=f(r,t)$ and $v_i(x)=u_i(r)$ with $i=0,1,$ respectively.
 Then  \eqref{sprciprob2} can  be  reformulated  as
\begin{equation}\label{sprciprob3}
\begin{split}
& \frac{\partial^2 v}{\partial t^2}-\frac {\tilde c^2} {(x+c_0)^{d-1}}
\frac{\partial}{\partial x} \Big((x+c_0)^{d-1}
\frac {\partial v}{\partial x}\Big)+ \tilde c^2\beta_n \frac v{(x+c_0)^2}=h,\quad  x\in I,\;\; t>0; \\
& v(x,0)=v_0(x),\quad \frac{\partial v}{\partial t}(x,0)=v_1(x), \quad  x\in I;
\quad  v(-1,t)=0,\;\; t>0; \\
  & \Big(\frac 1 c \frac {\partial v}{\partial t} + \frac 2 {b-b_0}\frac {\partial v}{\partial x}+\frac {d-1}{2b}v\Big)(1,t)=\int_0^t  \sigma_\nu(t-\tau)  v(1,\tau) d\tau, \;\; t>0,
\end{split}
\end{equation}
where the constants $\tilde c= \frac {2c}{b-b_0}$ and $c_0=\frac {b+b_0}{b-b_0}.$

Let $V_N:=\big\{\psi\in P_N : \psi(-1)=0\big\},$ where $P_N$ is the set of all  algebraic polynomials of degree at most  $N.$ The semi-discretization Legendre spectral-Galerkin  approximation of  \eqref{sprciprob3} is to find $v_N(x,t)\in V_N$ for all $t>0$ such that for all  $ w\in V_N,$
\begin{equation}\label{VarioneD}
\begin{split}
& (\ddot{v}_N, w)_\varpi  +\tilde c (1+c_0)^{d-1} \dot{v}_N(1,t) w(1) +\tilde c^2
(\partial_x v_N, \partial_x w)_{\varpi} +\tilde c^2 \beta_n\big(v_N(x+c_0)^{-2}, w\big)_{\varpi} \\
&\qquad + \frac {2c^2}{b-b_0}(1+c_0)^{d-1}\Big(\frac {d-1}{2b} v_N(1,t)-\sigma_\nu(t)\ast v_N(1,t)\Big) w(1)=
(I_N h,w)_{\varpi},\\
& v_N(x,0)=v_{0,N}(x),\quad \dot{v}_N(x,0)=v_{1,N}(x),\quad x\in I,
\end{split}
\end{equation}
where $(\cdot,\cdot)_{\varpi}$ is the (weighted) inner product of $L^2_{\varpi}(I)$ with the weight function
 $\varpi=(x+c_0)^{d-1},$  $\ddot{v}$ denotes $\partial_{t}^2v$ or $\frac{d^2v}{dt^2}$ as usual,
 $I_N$ is the interpolation operator on $(N+1)$ Legendre-Gauss-Lobatto points,  and
 $v_{0,N}, v_{1,N}\in P_N$ are suitable approximations of the initial values.

Like  Theorem  \ref{stabilcont}, we have the following {\em a prior}
estimates.
\begin{theorem}\label{aprior}   Let $v_N$ be the solution of \eqref{VarioneD}. Then for all $ t>0,$
 \begin{equation}\label{condit}
 \begin{split}
& \|\partial_t v_N\|_{L^2_{\varpi}(I)}^2 +
 \tilde c^2\big(\|\partial_x v_N\|^2_{L^2_\varpi(I)}+\beta_n \big\|v_N/(x+c_0)\big \|^2_{L^2_\varpi(I)} \big) \\
 &\quad \le C\Big(  \|v_{1,N}\|_{L^2_{\varpi}(I)}^2 +
 \tilde c^2\big(\|\partial_x v_{0,N}\|^2_{L^2_\varpi(I)}+\beta_n \big\|v_{0,N}/(x+c_0)\big \|^2_{L^2_\varpi(I)} \big)+
 \|I_Nh\|^2_{L^2_\varpi(I)} \Big),
 \end{split}
 \end{equation}
 where $C$ is a positive constant independent of $N, \tilde c$ and $b.$

 This estimate holds for   \eqref{sprciprob3} with $v, v_0, v_1, h$ in place of $v_N, v_{0,N}, v_{1,N}, I_Nh,$ respectively.
 \end{theorem}
 \begin{proof} Taking $w=\partial_t v_N$ in \eqref{VarioneD}, and integrating the resulted equation with respect to $t,$
 we use Corollary \ref{sigma} and the argument similar to  that for Theorem  \ref{stabilcont}  to derive the estimates.
 \end{proof}

 \begin{remark}\label{newadds}  Equipped with Theorem {\rm \ref{aprior}},  we can analyze the convergence of the
 semi-discretized  scheme  \eqref{VarioneD} as on Page {\rm 341} of \cite{ShenTaoWang2011}.   \qed
 \end{remark}

We next examine the linear system of   \eqref{VarioneD}.    As shown in \cite{Shen94b,ShenTaoWang2011},
it is advantageous to construct basis functions satisfying the underlying homogeneous Dirichlet boundary conditions. Let
 $L_l(x)$ be the Legendre polynomial of degree $l$ (see, e.g., \cite{szeg75}), and define
  $\phi_k(x)=L_{k}(x)+L_{k+1}(x).$ Then  $\phi_k(-1)=0$ and $V_N={\rm span}\big\{\phi_k : 0\le k\le N-1\big\}.$  Note that  $\phi_k(1)=2.$
   Setting
  \begin{equation*}\label{vnwnmatrix}
  \begin{split}
  & v_N(x,t)=\sum_{j=0}^{N-1} \hat v_j(t) \phi_j(x),\;\; {\pmb v}(t)=(\hat v_0, \cdots, \hat v_{N-1})^t,\;\;
  m_{ij}=(\phi_j,\phi_i)_{\varpi},\;\; s_{ij}=(\phi_j', \phi_i')_{\varpi}, \\
  \end{split}
  \end{equation*}
  we obtain the system:
  \begin{equation}\label{Lspsystem}
  \begin{split}
&M \ddot{\pmb v}+ \alpha  E \dot{\pmb v}+ \tilde c^2\big(S+\beta_n
\widetilde M\big){\pmb v}+ \mu  E {\pmb v}  -c\alpha\,
\Big(\sigma_\nu\ast \sum_{j=0}^{N-1} \hat v_j\Big){\pmb 1}={\pmb h},
  \end{split}
  \end{equation}
  with $ \pmb v(0)=\pmb v_0$ and $\dot{\pmb v}(0)=\pmb v_1,$
  where $M=(m_{ij}), S=(s_{ij}), \widetilde M=(\tilde m_{ij})$
  and $E={\pmb 1} {\pmb 1}^t$ is an $N\times N$ matrix of all ones.  In \eqref{Lspsystem},  $\pmb v_0$ and  $\pmb v_1$
  are column-$N$ vectors of the expansion coefficients of $v_{0,N}$ and $v_{1,N}$ in terms of $\{\phi_k\},$ and the constants $\alpha= {8c}(1+c_0)^{d-1}/(b-b_0)$ and $\mu= {4c^2(d-1)}(1+c_0)^{d-1}/(b(b-b_0)).$

\subsection{Newmark's time integration}\label{Sect:Newmark} To discretize  the second-order ordinary differential system \eqref{Lspsystem},  we resort to the implicit second-order Newmark's scheme,
which has wide applications in the field of structural mechanics
(see \cite{newmark1959,Wood1990}).
To this end, let $\Delta t$ be the
time-stepping size, and let $\{\pmb v^m, \dot{\pmb v}^m, \ddot{\pmb
v}^m\}$ be the approximation of $\{\pmb v, \dot{\pmb v}, \ddot{\pmb
v}\}$ at $t=t_m=m\Delta t,$ and $\pmb h^m=\pmb h(t_m).$

We first take care of  the convolution term
$\big[\sigma_\nu\ast \sum_{j=0}^{N-1} \hat v_j\big](t_{m+1})$ in
\eqref{Lspsystem}:
 \begin{equation*}\label{convappro}
 \begin{split}
 &\Big[\sigma_\nu\ast \sum_{j=0}^{N-1} \hat v_j\Big](t_{m+1})
 =\sum_{j=0}^{N-1} \int_{t_m}^{t_{m+1}} \sigma_{n}(t_{m+1}-\tau)  \hat v_j(\tau) d\tau+
 \sum_{j=0}^{N-1} \int_{0}^{t_{m}} \sigma_{n}(t_{m+1}-\tau)  \hat v_j(\tau) d\tau\\
 &\qquad \approx  \frac  {\Delta t} 2\Big(\sigma_\nu(0) \sum_{j=0}^{N-1}\hat v_j^{m+1}+ \sigma_\nu(\Delta t) \sum_{j=0}^{N-1}\hat v_j^{m}\Big)+ \sum_{j=0}^{N-1} \int_{0}^{t_{m}} \sigma_{n}(t_{m+1}-\tau)  \hat v_j(\tau) d\tau\\
 &\qquad \approx  \frac  {\Delta t} 2  \sigma_\nu(0)E{\pmb v}^{m+1}+ {\pmb g}^m,
  \end{split}
 \end{equation*}
 where we used the Trapezoidal rule (of order $O(\Delta t^3)$) to approximate the integral over $(t_m, t_{m+1}),$ and   denoted by ${\pmb g}^m$ the approximation of the remaining terms. Note that ${\pmb g}^m$ depends on the history
 $\pmb v^l (0\le l\le m),$ but fortunately, it can be evaluated recursively and rapidly as described in
 Subsection \ref{Sect:rapideva}. Moreover,  only the history data $\sum \hat v_j^l\, (0\le l\le m)$  need to be stored, and no any matrix-vector multiplication is involved.
 As a result, the burden of the history dependence can be eliminated.

Denoting
\begin{equation}\label{newmarkmatrices}
A=M,\;\;  B=\alpha E,\;\; C=\tilde c^2\big(S+\beta_n \widetilde
M\big) + \Big(\mu -c\alpha \frac  {\Delta t} 2  \sigma_\nu(0)\Big)E,
\end{equation}
and $\pmb f^{m+1}=\pmb h^{m+1}+c\alpha {\pmb g}^m,$
we carry out the
full scheme for \eqref{sprciprob3} as follows.
\begin{itemize}
\item[(i)] Set $\pmb v^0=\pmb v_0$ and $\dot{\pmb v}^0=\pmb v_1,$ and  compute $\ddot{\pmb v}^0$ from the system \eqref{Lspsystem} by
\begin{equation}\label{step1}
\ddot{\pmb v}^0=M^{-1} \big\{\pmb h^0- \alpha  E \dot{\pmb v}^0-
\tilde c^2\big(S+\beta_n \widetilde M\big){\pmb v}^0- \mu  E {\pmb
v}^0 \big\},
\end{equation}
where  the convolution term vanishes at $t=0.$

\item[(ii)] For $m\ge 0$, compute ${\pmb v}^{m+1}$ from
\begin{equation}\label{genexampleaest}
\begin{split}
(A+\vartheta \Delta t B+\theta \Delta t^2 C\big) \ddot{\pmb
v}^{m+1}&=
{\pmb f}^{m+1}- B\big(\dot{\pmb v}^m +(1-\vartheta) \Delta t \ddot{\pmb v}^m\big)\\
&\;\; - C\Big(\pmb v^{m}+\Delta t \dot{\pmb v}^m +\frac{1-2\theta}
2\Delta t^2 \ddot{\pmb v}^m\Big),
\end{split}
\end{equation}
and update  ${\pmb u}^{m+1}$ and $\dot {\pmb u}^{m+1}$ by
\begin{align}
&\pmb v^{m+1}=\pmb v^{m}+\Delta t \dot{\pmb v}^m +\frac{1-2\theta} 2\Delta t^2 \ddot{\pmb v}^m+\theta\Delta t^2 \ddot{\pmb v}^{m+1}, \label{genexamplea0}\\
&\dot{\pmb v}^{m+1}=\dot{\pmb v}^m +(1-\vartheta) \Delta t
\ddot{\pmb v}^m+\vartheta\Delta t \ddot{\pmb v}^{m+1},
\label{genexamplea00}
\end{align}
respectively.
\end{itemize}
\begin{remark}\label{Augaddw} In general,
the  Newmark's scheme is of second-order and unconditionally
stable, if  the parameters satisfy $\vartheta \ge \frac 1 2$
and $\theta\ge \frac 1 4\big(\frac 1 2 +\vartheta\big)^2$ {\rm(}see, e.g., \cite{Wood1990}{\rm)}.

The computational  cost of the full algorithm lies in solving \eqref{genexampleaest} with \eqref{newmarkmatrices}.
The matrices $M, \widetilde M$ and $S$ are sparse with small bandwidth under the basis $\{\phi_k\},$ and their entries can be evaluated exactly by using the properties of Legendre polynomials. The matrices  $B$ and $C$ in  \eqref{newmarkmatrices} appear to be full, but since  all entries of $E$ are one,  the matrix-vector multiplication    can be carried out in $O(N)$ operations.  \qed
\end{remark}

\section{Numerical results}\label{sect5}
\setcounter{equation}{0}

In this section, we present various numerical results to show the accuracy of the NRBC and
and convergence of the proposed algorithm.

\subsection{Testing problem and setup}
We examine  the NRBC and the scheme via the following problem with  an exact solution.
\begin{proposition}\label{exact}
The exterior  problem \eqref{eq1.1a}-\eqref{eq1.1c} with $d=2,
D=\{\bs x\in {\mathbb R^2} : |\bs x|<b_0\}, F=U_0=U_1\equiv 0,$ and
$U|_{r=b_0}=G,$ admits
 the  solution:
\begin{equation}\label{exactsln}
U(r,\phi,t)=\sum_{|n|=0}^\infty \widehat U_n(r,t)e^{\ri n\phi},\quad r>b_0, \; \phi\in [0,2\pi), \; t>0,
\end{equation}
where
\begin{equation}\label{exactslncoef}
\widehat U_n(r,t)=
\begin{cases}
0, & \quad t<\beta_0,\\
H_n(r,t)\ast\widehat G_n(t-\beta_0)+\sqrt{b_0/r}\widehat
G_n(t-\beta_0),&\quad t\ge \beta_0,
\end{cases}
\end{equation}
with $\beta_0=(r-b_0)/c,$ $\widehat G_n (t)$ being the Fourier
expansion coefficient of $G(\phi,t),$ and
\begin{equation}\label{exactslncoef1}
\begin{split}
H_n(r,t)&=\frac{c}{b_0}\sum_{j=1}^{M_n}\frac{K_{n+1}(rz_j^n/{b_0})}{K'_n(z_j^n)}e^{t cz_j^n/b_0}\\
&+(-1)^n\frac{c}{b_0}\int_0^\infty\frac{I_n(r\rho/b_0)K_n(\rho)-K_n(r\rho/b_0)I_n(\rho)}{K^2_n(\rho)+\pi^2I^2_n(\rho)}e^{-ct\rho/b_0}
d\rho.
\end{split}
\end{equation}
\end{proposition}
\begin{proof}
We sketch its derivation  in Appendix \ref{justappC}. \end{proof}

\begin{remark}\label{remasckk}
 Notice that $\widehat U_n$ satisfies the one-dimension problem:
\begin{align}
& \frac{\partial^2 \widehat U_n}{\partial t^2}-\frac {c^2} r
\frac{\partial}{\partial r} \Big(r \frac {\partial  {\widehat
U_n}}{\partial r}\Big)+ c^2 n^2 \frac { \widehat U_n}{r^2}=0,\quad \;\;  b_0<r<b,\; t>0; \label{1deqna} \\
&  \widehat U_n|_{t=0}=\frac{\partial \widehat U_n}{\partial
t}\Big|_{t=0}=0, \quad  b_0< r< b; \quad \widehat U_n |_{r=b_0}=\widehat G_n,\;\;
t>0; \label{1deqnb}\\
& \Big(\frac 1 c \frac {\partial  \widehat U_n}{\partial t} + \frac
{\partial  \widehat U_n}{\partial r}+\frac {1}{2r} \widehat
U_n\Big)\Big|_{r=b}=\int_0^t  \sigma_n(t-\tau)  \widehat U_n (b,\tau)
d\tau, \;\; t>0.\label{1dbndry} \qed
\end{align}
\end{remark}

We take the Dirichlet data:
\begin{equation}\label{example1}
G(\phi,t)=A_1e^{-\iota ((b_0\cos\,\phi-x_s)^2+(b_0\sin\,\phi-y_s)^2)}\sin^p(\omega
t),
\end{equation}
where  $A_1=10, \iota=0.1, x_s=y_s=2.1, b_0=2,$ and $\omega, p$  are to  be specified later.
We compute the Fourier coefficient
\[
\widehat G_n(t)=\frac { A_1}{2\pi}\sin^p(\omega t)\int_0^{2\pi} e^{-\iota ((b_0\cos\,\phi-x_s)^2+(b_0\sin\,\phi-y_s)^2)}e^{-\ri n\phi}d\phi,
\]
via the fast Fourier transform with  sufficient  Fourier points so that the error of numerical integration  can be  neglected.

The first important issue is to compute the exact solution and $\sigma_n$ very accurately.  We point out that  the convolution  in
\eqref{exactslncoef} can be computed exactly, and  the improper integral in \eqref{exactslncoef1} can be treated by a
similar manner as the improper integral of $\sigma_n(t)$ in Subsection  \ref{impComLab}.
We adopt a Newton's iteration method to compute the zeros of $K_n(z)$  accurately for moderate large $n,$ and use
a (composite) Legendre-Gauss quadrature  to evaluate the integral with  high precision.

For the readers' reference, we tabulate in  Table \ref{tb7} some samples  of $\sigma_{n}(t)$ with $b=3, c=5.$
\begin{table}[!ht]
{\footnotesize
\begin{center}
 \caption{\small Some samples of  $\sigma_n(t).$}
 \vspace*{-8pt}
\begin{tabular}{|c|c|c|c|c|c|}
\hline
$n$  & $\sigma_n(0.1)$ & $\sigma_n(2)$ & $n$  & $\sigma_n(0.1)$ & $\sigma_n(2)$ \\
\hline
0  &  \,\,5.922023678764810e-2   &      \,\,1.127355852971518e-2    &5   & -5.394296605508512 &   \,\,2.652114252130534e-3\\
\hline
1  &    -1.770775252065292e-1    &     -1.849758830245009e-2        & 6  & -7.505798337812085  &       -1.628640292879232e-3    \\
\hline
2  &     -8.766785435408012e-1  &     -3.516167498630298e-3         & 7  & -9.786660490985828  &     \,\,1.032611050422504e-3   \\
\hline
3  & -2.012014067208546         &      \,\,6.186828109004313e-3     & 8  & -12.14241492386835  &       -6.793719076424341e-4 \\
\hline
4  & -3.538153424430011         &     -4.319820513209325e-3         & 9  & -14.47330651171318  &     \,\,4.628600071971645e-4 \\
\hline
\end{tabular}
\end{center}\label{tb7}
}
\end{table}

We plot in Figure \ref{Afigrr2} the exact solution with $p=2, b_0=2, b=4,\omega=10\pi$  at various $t.$ 
We see that the Dirichlet data $G(\phi,t)$ acts as a dynamical wave-maker.

\begin{figure}[!th]
  \begin{center}
    \includegraphics[width=0.28\textwidth]{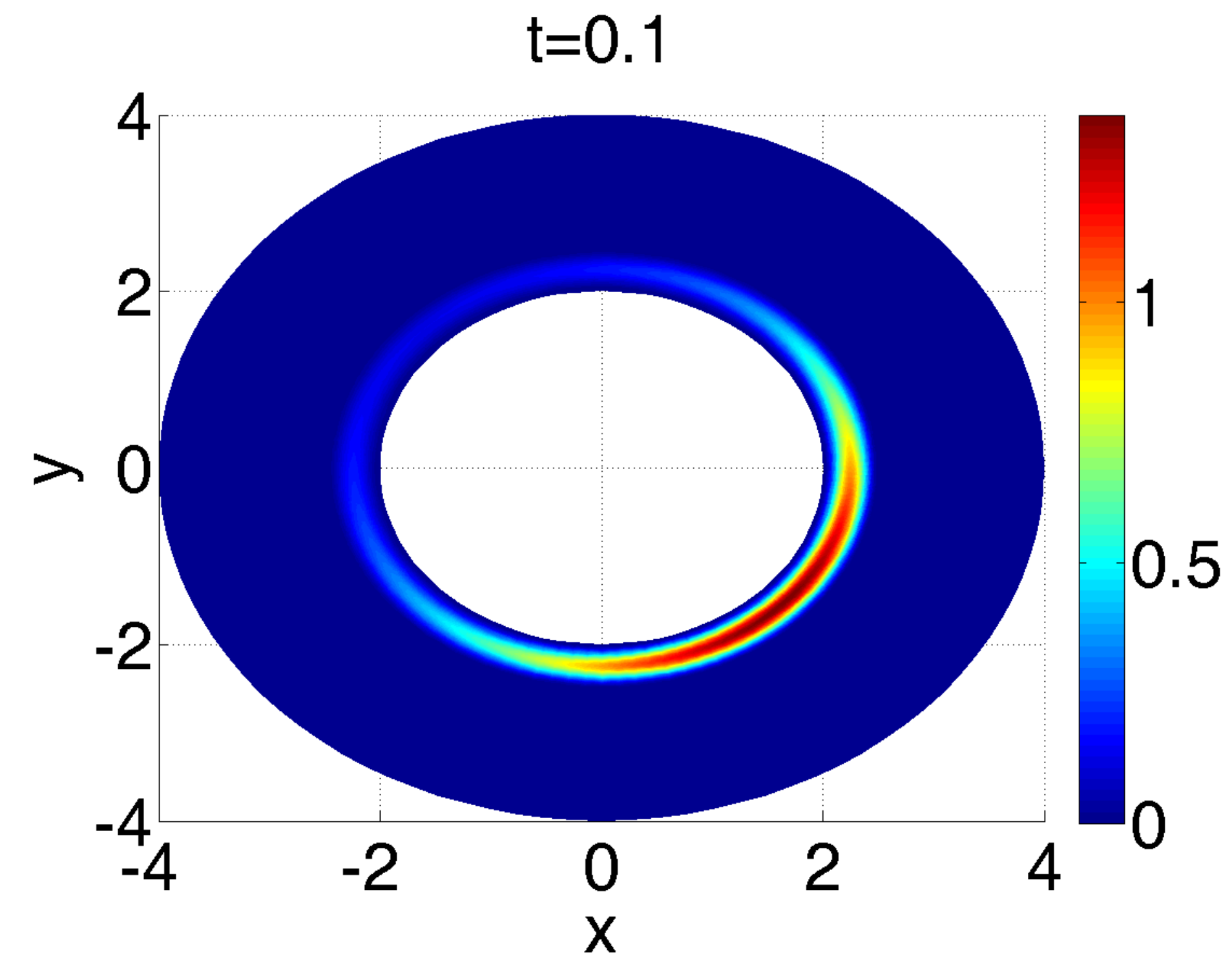}
    \includegraphics[width=0.28\textwidth]{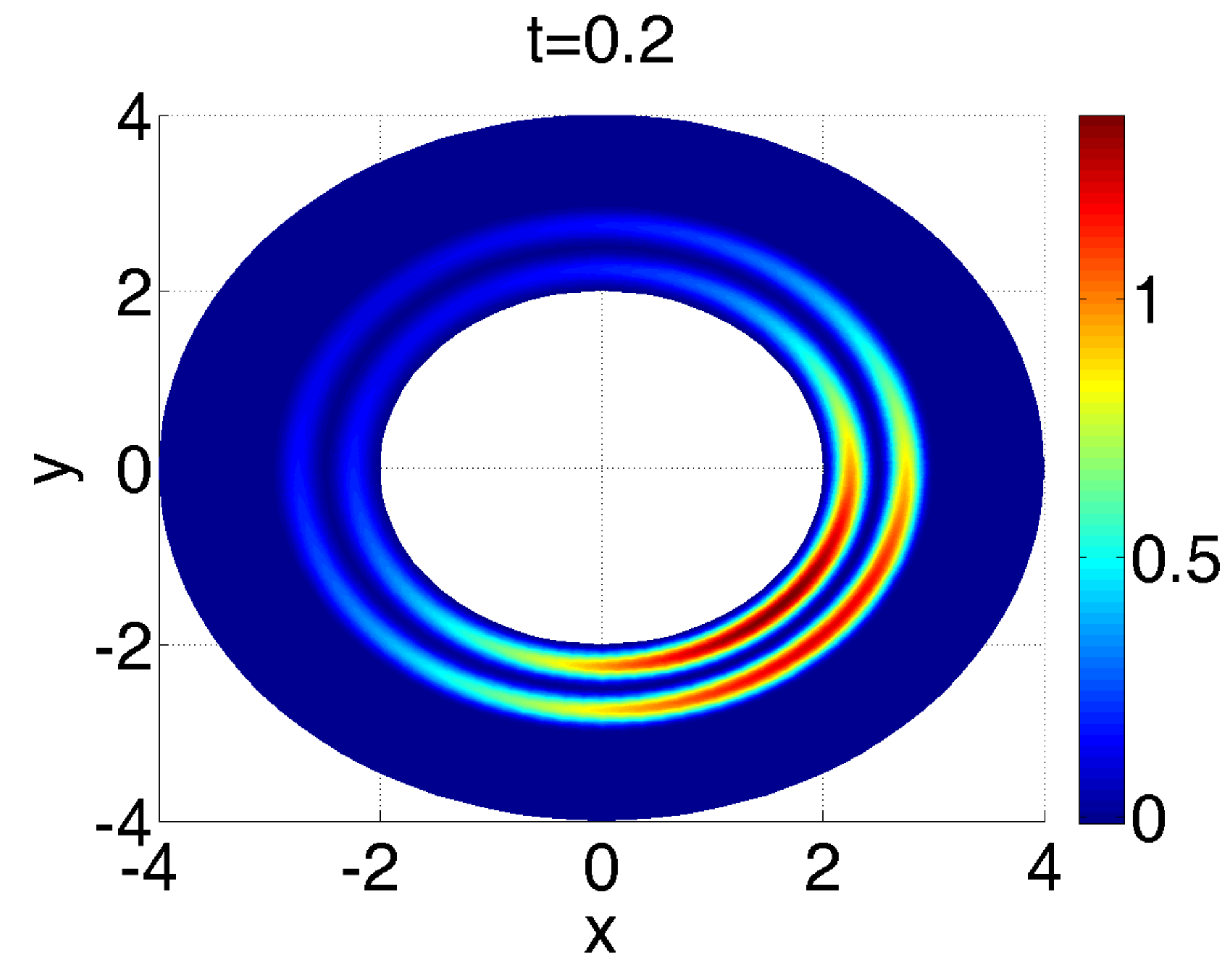}
    \includegraphics[width=0.28\textwidth]{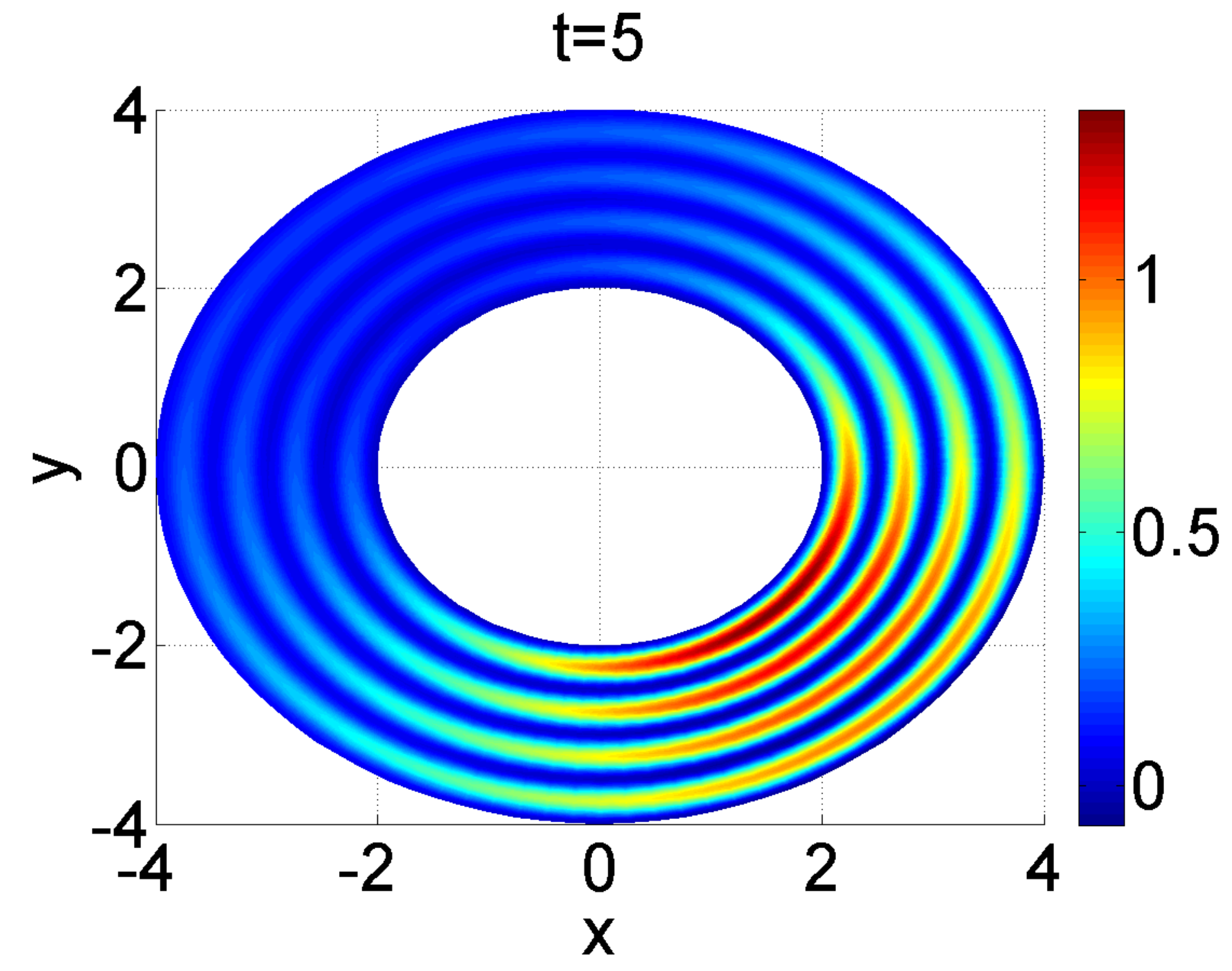}
\caption{\small Time evolution of exact solution with $\omega=10\pi$.}\label{Afigrr2}
  \end{center}
  \end{figure}

\subsection{Accuracy of the NRBC}
Now, we examine   accuracy of the  NRBC. Notice that
\begin{equation*}\label{err1}
\max_{|\phi|\le 2\pi}\big|\partial_r U_M-T_d(U_M)\big|\le
\sum_{|n|=0}^M e_n(t),\; e_n(t):= \Big|\Big(\frac 1 c \frac {\partial \widehat U_n}{\partial
t} + \frac {\partial \widehat U_n }{\partial r} +\frac {\widehat
U_n}{2r}\Big)\Big|_{r=b}-\sigma_n(t)\ast\widehat U_n(b,t)\Big|,
\end{equation*}
where $T_d(U_M)$ is defined in \eqref{GdUadd} with $d=2,$  and
$U_M(b,\phi,t)=\sum_{|n|=0}^M \widehat U_n(b,t) e^{\ri n \phi}$
is the truncation of the exact solution in Proposition  \ref{exact}.  We choose the same parameters as those
for the exact solution in Figure  \ref{Afigrr2}, and take $M=32$ so that $\{|\widehat U_n(b,t)|\}$ are sufficiently small for
all modes $|n|\le M$ and all $t$ of interest.   Note that the differentiations in
$e_n(t)$ can be performed exactly  on the exact solution. 

In Table \ref{tb8}, we tabulate the errors:
\[
E_M^{(1)}(t)=\max_{|n|\le M} e_n(t),\quad E_M^{(2)}(t)=\sum_{|n|=0}^M e_n(t),
\]
at some samples of $t,$  compute the errors at the outer artificial boundary $r=b,$ where the exact solution has the magnitude as large as possible.

\begin{table}[!ht]
{\footnotesize
\begin{center}
 \caption{\small The errors $E_M^{(1)}(t)$ and $E_M^{(2)}(t).$  }
 \vspace*{-8pt}
\begin{tabular}{|c|c|c|c|c|c|c|c|c|}
\hline \multicolumn{1}{|c|}{$t$}&
\multicolumn{2}{c|}{$\omega=10\pi,b=2.22$}&
\multicolumn{2}{c|}{$\omega=10\pi, b=2.75$} &
\multicolumn{2}{c|}{$\omega=20\pi, b=2.38$}&
\multicolumn{2}{c|}{$\omega=20\pi, b=2.87$}
\\
\cline{2-9}
 & $E_M^{(1)}(t)$   & $E_M^{(2)}(t)$ & $E_M^{(1)}(t)$
 & $E_M^{(2)}(t)$   & $E_M^{(1)}(t)$& $E_M^{(2)}(t)$
 & $E_M^{(1)}(t)$   & $E_M^{(2)}(t)$ \\
\cline{1-9}
0.5 &    2.637e-16   &  4.902e-16  &  7.026e-17  &  1.148e-16   &  6.245e-17   &  1.603e-16   &  8.327e-17   &  1.752e-16 \\
\hline
1   &    4.684e-16   &  8.406e-16  &  2.099e-16  &  2.630e-16   &   9.021e-17  &  2.390e-16   &  2.272e-16   &  4.405e-16\\
 \cline{1-9}
5   &    2.567e-16   &  4.386e-16  &  1.131e-15  &  1.176e-15   &  7.251e-16   &  9.053e-16   &  1.318e-15   &  1.473e-15\\
 \cline{1-9}
10  &    7.772e-16   &  1.020e-15  &  2.207e-15  &  2.227e-15   &  1.457e-15   &  1.960e-15   &  2.387e-15   &  2.575e-15\\
\hline
\end{tabular}
\end{center}\label{tb8}
}
\end{table}
We see from Table \ref{tb8} that in all tests, the errors are extremely small. This validates the formula  for $\sigma_n$ in Theorem \ref{theorem1}  and high accuracy of the  numerical treatment.

\subsection{Numerical tests for the spectral-Galerkin-Newmark's scheme}
Hereafter, we provide some numerical results to illustrate  the convergence  of the numerical scheme described in Section \ref{sect4}.    In the following computations, we take the Dirichlet data given by  \eqref{example1}
with $A_1=10, \iota=0.1, x_s=y_s=2.1$ and $b_0=2$ as before.

 Given a cut-off number  $M>0,$ we compute the numerical solutions $\{\widehat U_n^N\}$ of
   \eqref{1deqna}-\eqref{1dbndry} for the modes $|n|\le M$ by using the spectral-Galerkin
 and Newmark's time integration scheme. The full numerical solution of the problem in Proposition
 \ref{exact} is then defined as  $U_{M}^N(r,\phi,t)=\sum_{|n|=0}^{M} \widehat{U}_n^N(r,t)e^{\ri
n\phi},$ and the exact solution
$U_M(r,\phi,t)=\sum_{|n|=0}^{M} \widehat{U}_n(r,t)e^{\ri n\phi}$  is evaluated as before.  Once again, we choose
$M$ such that the Fourier coefficient of the exact solution $|\widehat{U}_n|\le 10^{-16}$ for $|n|>M.$
 The numerical errors are measured by
  \[
\widehat{E}_M^N(t)=\max_{|n|\le M}\big\|\widehat{U}_n^N(\cdot,t)-\widehat{U}_n(\cdot,t)\big\|_{L^2, N},
\quad \widetilde{E}_M^N(t)=\max_{|n|\le M} \big\|\widehat{U}_n^N(\cdot,t)-\widehat{U}_n(\cdot,t)\big\|_{L^\infty,N},
\]
where $\|\cdot\|_{L^2,N}$ is the discrete $L^2$-norm associated with the Legendre-Gauss-Lobatto interpolation,  and $\|\cdot\|_{L^\infty,N}$ is the corresponding discrete maximum norm.

To test the second-order convergence of the Newmark's scheme,  we
choose $N=50$ so that the error of the spatial discretization is
negligible. We provide in Table \ref{tb9a}  the numerical errors and
the order of convergence for various $t$ with $M=15, b=5,
\omega=\pi,$ and $p=6.$  As expected, we observe a second-order
convergence of the time integration.
\begin{table}[!ht]
{\footnotesize
\begin{center}
 \caption{\small The errors $\widehat{E}_M^N(t)$ and $\widetilde{E}_M^N(t),$ and
order of convergence.}
\vspace*{-8pt}
\begin{tabular}{|c|c|c|c|c|c|c|c|c|c|c|c|}
\hline \multicolumn{1}{|c|}{$t$}&\multicolumn{1}{|c|}{$\Delta t$}&
\multicolumn{1}{c|}{$\widehat{E}_M^N(t)$}&
\multicolumn{1}{c|}{order} &
\multicolumn{1}{c|}{$\widetilde{E}_M^N(t)$}&
\multicolumn{1}{c|}{order}&\multicolumn{1}{|c|}{$t$}&\multicolumn{1}{|c|}{$\Delta
t$}& \multicolumn{1}{c|}{$\widehat{E}_M^N(t)$}&
\multicolumn{1}{c|}{order} &
\multicolumn{1}{c|}{$\widetilde{E}_M^N(t)$}&
\multicolumn{1}{c|}{order}
\\
\cline{1-12}
     & 1e-03  & 1.21e-05 &         & 1.93e-05 &            &      & 1e-03 & 1.23e-05 &         & 2.01e-05 &       \\
\cline{2-6} \cline{8-12}
 1   & 5e-04  & 3.02e-06 & 2.00005 & 4.83e-06 &  1.99994   &  3   & 5e-04 & 3.07e-06 & 2.00007 & 5.02e-06 & 1.99997 \\
\cline{2-6} \cline{8-12}
     & 1e-04  & 1.21e-07 & 1.99993 & 1.93e-07 &  1.99986   &      & 1e-04 & 1.23e-07 & 1.99997 & 2.01e-07 & 1.99998 \\
\cline{2-6} \cline{8-12}
     & 5e-05  & 3.02e-08 & 2.00012 & 4.83e-08 & 1.99906    &      & 5e-05 & 3.07e-08 & 2.00007 & 5.02e-08 & 2.00004 \\
\cline{2-12}
     & 1e-03 & 1.23e-05  &         & 2.01e-05 &            &      & 1e-03 & 1.23e-05 &         & 2.01e-05 &          \\
\cline{2-6} \cline{8-12}
 2   & 5e-04  & 3.07e-06 & 1.99998 & 5.02e-06 & 1.99997    &  4   & 5e-04 & 3.07e-06 & 2.00007 & 5.02e-06 & 1.99997 \\
\cline{2-6} \cline{8-12}
     & 1e-04  & 1.23e-07 & 2.00001 & 2.01e-07 & 1.99998    &      & 1e-04 & 1.23e-07 & 1.99997 & 2.01e-07 & 1.99998\\
\cline{2-6} \cline{8-12}
     & 5e-05  & 3.07e-08 & 2.00002 & 5.02e-08 & 1.99993    &      & 5e-05 & 3.07e-08 & 2.00007 & 5.02e-08 & 2.00004\\
\cline{2-12} \hline
\end{tabular}
\end{center}\label{tb9a}
}
\end{table}

Next, we fix the time step size $\Delta t=10^{-5}$ and choose different $N$ to check the accuracy in spatial discretization. The convergence behavior is illustrated in Table \ref{tb9b}.  With a small
number of modes in space, we observe a fast decay of the errors, which is typical for the spectral approximation.

\begin{table}[!ht]
{\footnotesize
\begin{center}
 \caption{\small The errors $\widehat{E}_M^N(t)$ and $\widetilde{E}_M^N(t).$}
 \vspace*{-8pt}
\begin{tabular}{|c|c|c|c|c|c|c|c|c|}
\hline \multicolumn{1}{|c|}{$t$}& \multicolumn{2}{c|}{$N=8$}&
\multicolumn{2}{c|}{$N=10$}& \multicolumn{2}{c|}{$N=16$}&
\multicolumn{2}{c|}{$N=32$}
\\
\cline{2-9}
 & $\widehat{E}_M^N(t)$     & $\widetilde{E}_M^N(t)$ & $\widehat{E}_M^N(t)$
 & $\widetilde{E}_M^N(t)$   & $\widehat{E}_M^N(t)$   & $\widetilde{E}_M^N(t)$ & $\widehat{E}_M^N(t)$& $\widetilde{E}_M^N(t)$\\
 \hline
0.5  & 2.349e-04  & 2.687e-04  &   2.253e-05   &  2.846e-05 & 8.224e-07  & 1.207e-06 & 9.229e-09 &  1.639e-08  \\
\hline
1.0  & 3.877e-04  & 4.179e-04  &   1.708e-05   &  2.042e-05 & 1.833e-07  & 2.777e-07 & 1.959e-09 &  2.673e-09 \\
\hline
1.5  & 3.276e-04  & 3.596e-04  &   5.965e-06   &  7.146e-06 & 3.072e-08  & 5.238e-08 & 1.508e-09 &  1.762e-09  \\
\hline
2.0  & 4.105e-04  & 4.368e-04  &   1.092e-05   &  1.188e-05 & 7.353e-09  & 1.238e-08 & 1.237e-09 &  2.119e-09  \\
\hline
2.5  & 3.239e-04  & 3.607e-04  &   5.724e-06   &  6.997e-06 & 2.381e-09  & 3.182e-09 & 1.488e-09 &  1.734e-09  \\
\hline
3.0  & 4.113e-04  & 4.380e-04  &   1.095e-05   &  1.189e-05 & 1.404e-09  & 2.893e-09 & 1.227e-09 & 2.009e-09   \\
\hline
3.5  & 3.238e-04  & 3.608e-04  &   5.747e-06   &  7.036e-06 & 1.554e-09  & 2.063e-09 & 1.488e-09 & 1.732e-09   \\
\hline
4.0  & 4.113e-04  & 4.380e-04  &   1.095e-05   &  1.189e-05 & 1.269e-09  & 2.061e-09 & 1.227e-09 & 2.007e-09   \\
\hline
\end{tabular}
\end{center}\label{tb9b}
}
\end{table}

\begin{figure}[!t]
\subfigure[]{
\begin{minipage}[t]{0.32\textwidth}
\centering
\rotatebox[origin=cc]{-0}{\includegraphics[width=1.85in, height=1.8in]{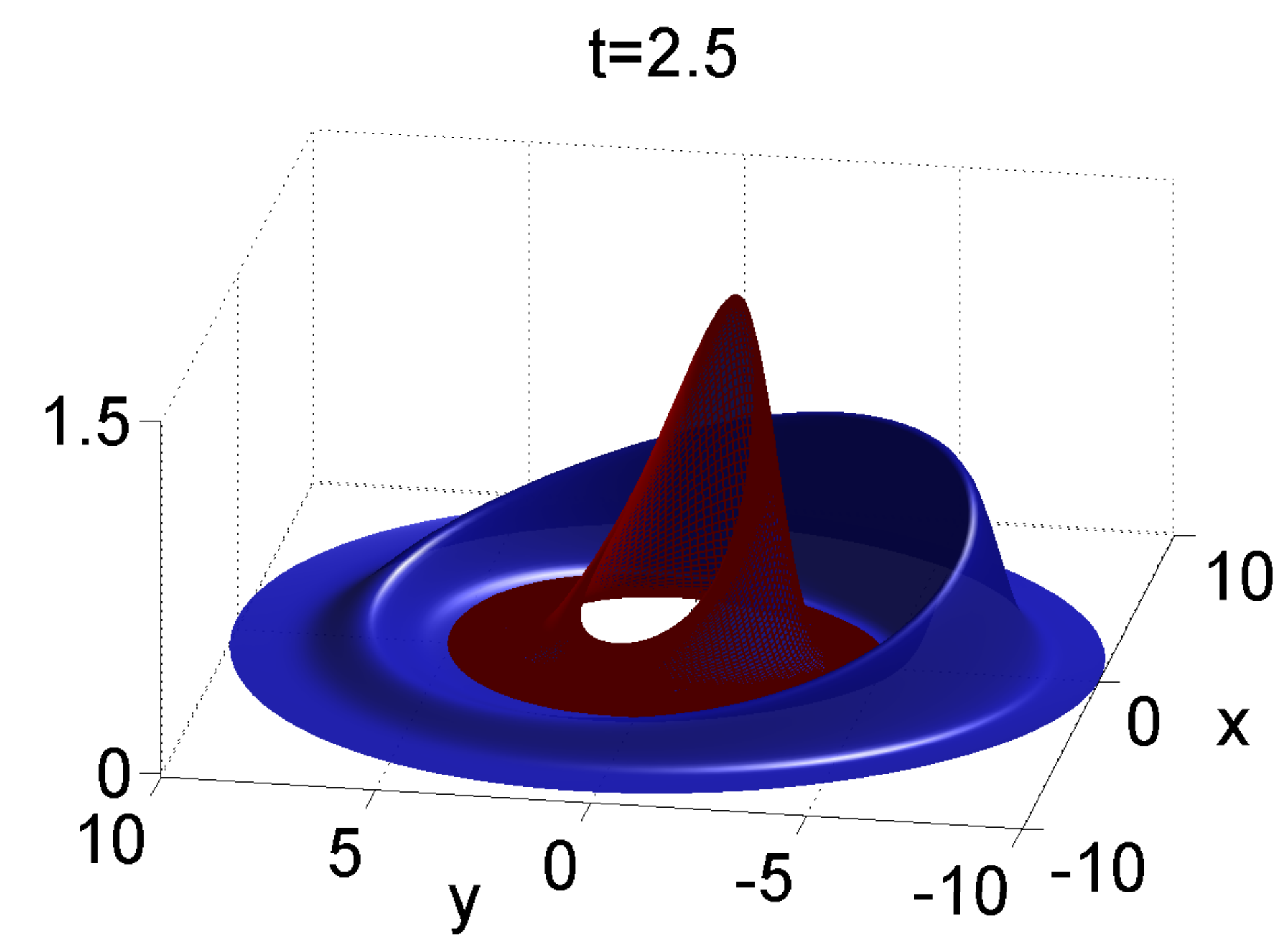}}
\end{minipage}}%
\subfigure[]{
\begin{minipage}[t]{0.32\textwidth}
\centering
\rotatebox[origin=cc]{-0}{\includegraphics[width=1.85in, height=1.8in]{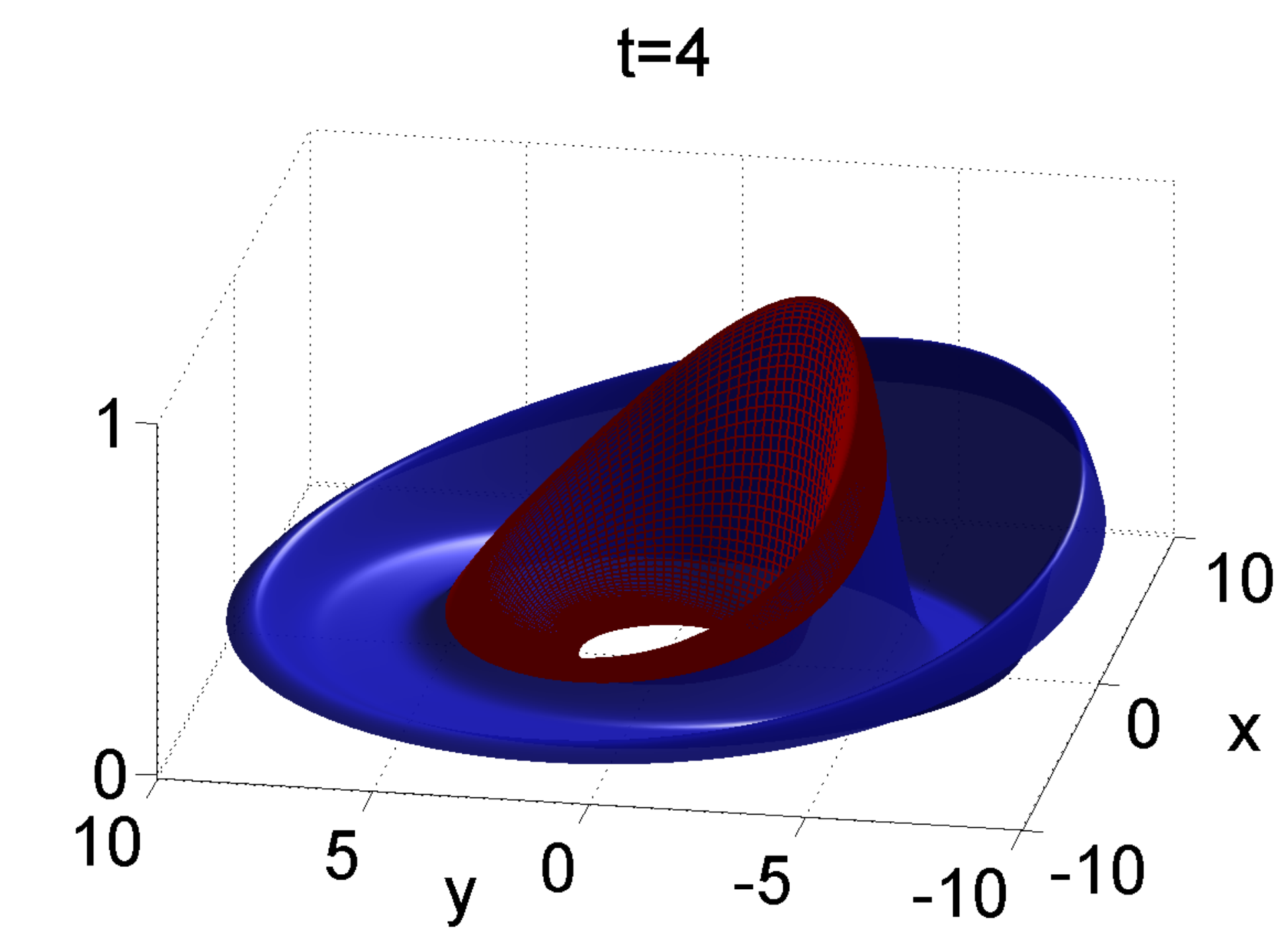}}
\end{minipage}}%
\subfigure[]{
\begin{minipage}[t]{0.32\textwidth}
\centering
\rotatebox[origin=cc]{-0}{\includegraphics[width=1.85in, height=1.8in]{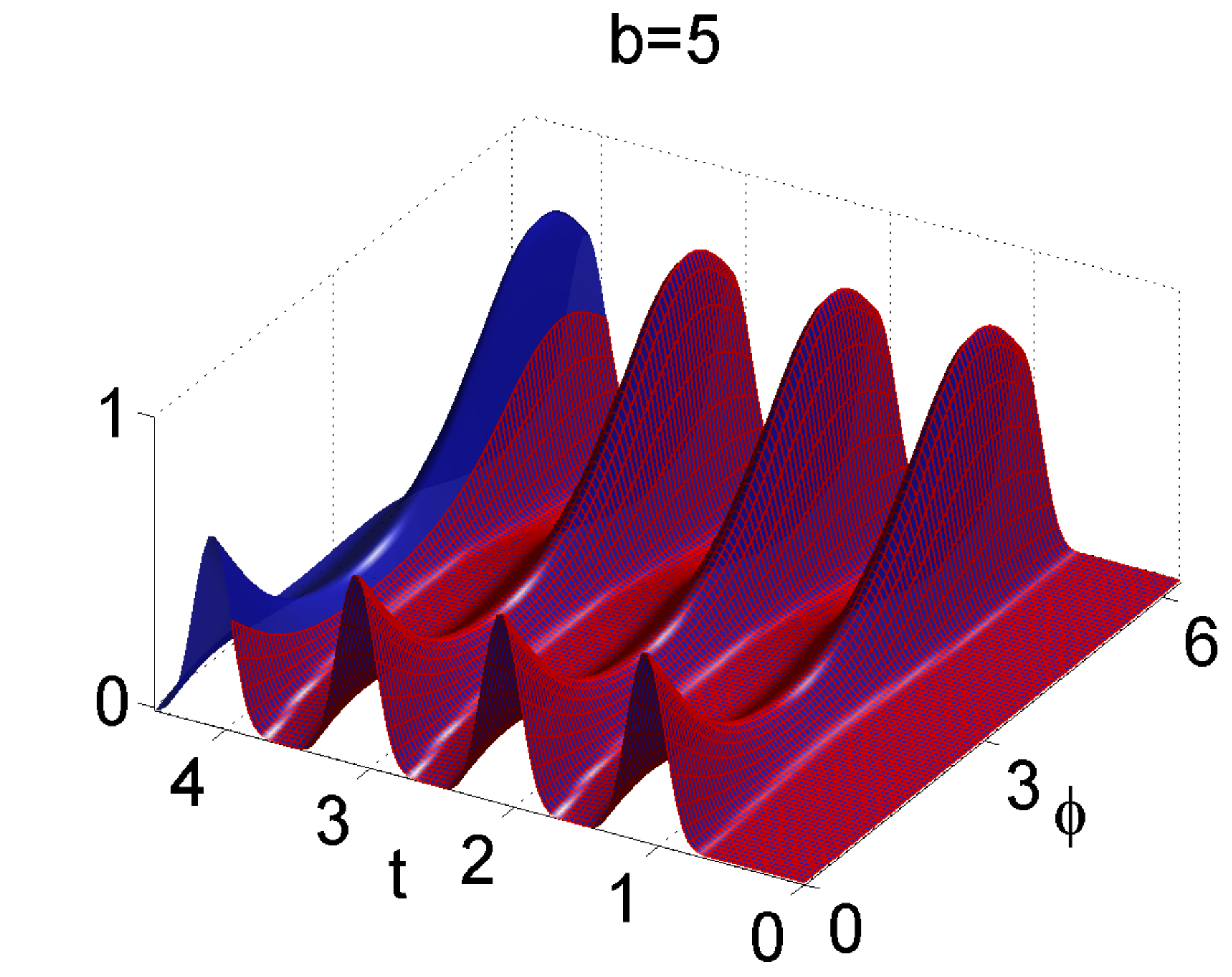}}
\end{minipage}}%
\caption{\small Propagation of exact  and numerical solutions.}
\label{fig2}
  \end{figure}

Finally, we plot in Figure  \ref{fig2} the numerical solution (red nets)   against the exact solution (blue smooth surface),
where the exact solution in (a) and (b) is in the annulus: $2\le r\le 10,$ while the numerical  solution within $2\le r\le 5.$
Figure \ref{fig2}(c) shows the wave propagation through the artificial boundary at $b=5, 0\le \phi\le 2\pi$ for
$0\le t\le 4$ (numerical solution) against      $0\le t\le 4.5$ (exact solution).  We see that the red nets and
blue surfaces  agree well in these plots. Moreover, we observe the waves pass the boundary transparently.
This shows that the proposed  scheme is very stable and has   high-order   accuracy.

\vskip 12pt

\noindent\underline{\large\bf Concluding remarks}

\vskip 5pt

We proposed in this paper analytic and accurate numerical means for
the time-domain wave propagation with exact and global nonreflecting
boundary conditions.  We derived the analytic expressions of the
convolution kernels and presented highly accurate methods for their
evaluations.  We analyzed the stability of the
truncated problem and provided efficient numerical schemes.  Ample
numerical results were given to demonstrate the features  of the
method. We shall report the combination of the methods with the
boundary perturbation technique in a future work.   The techniques 
 and ideas in this paper will be useful to study the Maxwell equations and elastic wave propagations.

\begin{appendix}
\section{Proof of Theorem \ref{theorem1}} \label{pffssl}
\renewcommand{\theequation}{A.\arabic{equation}}
\setcounter{equation}{0}

We first consider $d=2.$ Note that  $F_n(z)$ is a
multi-valued  function with branch points at $z=0$ and infinity. The
contour $L$ for  \eqref{inverlap} is
depicted in Figure \ref{contour1} (left) with the branch-cut along the negative real axis.

\vspace*{-6pt}
\begin{figure}[!ht]
\begin{center}
\includegraphics[width=0.55\textwidth]{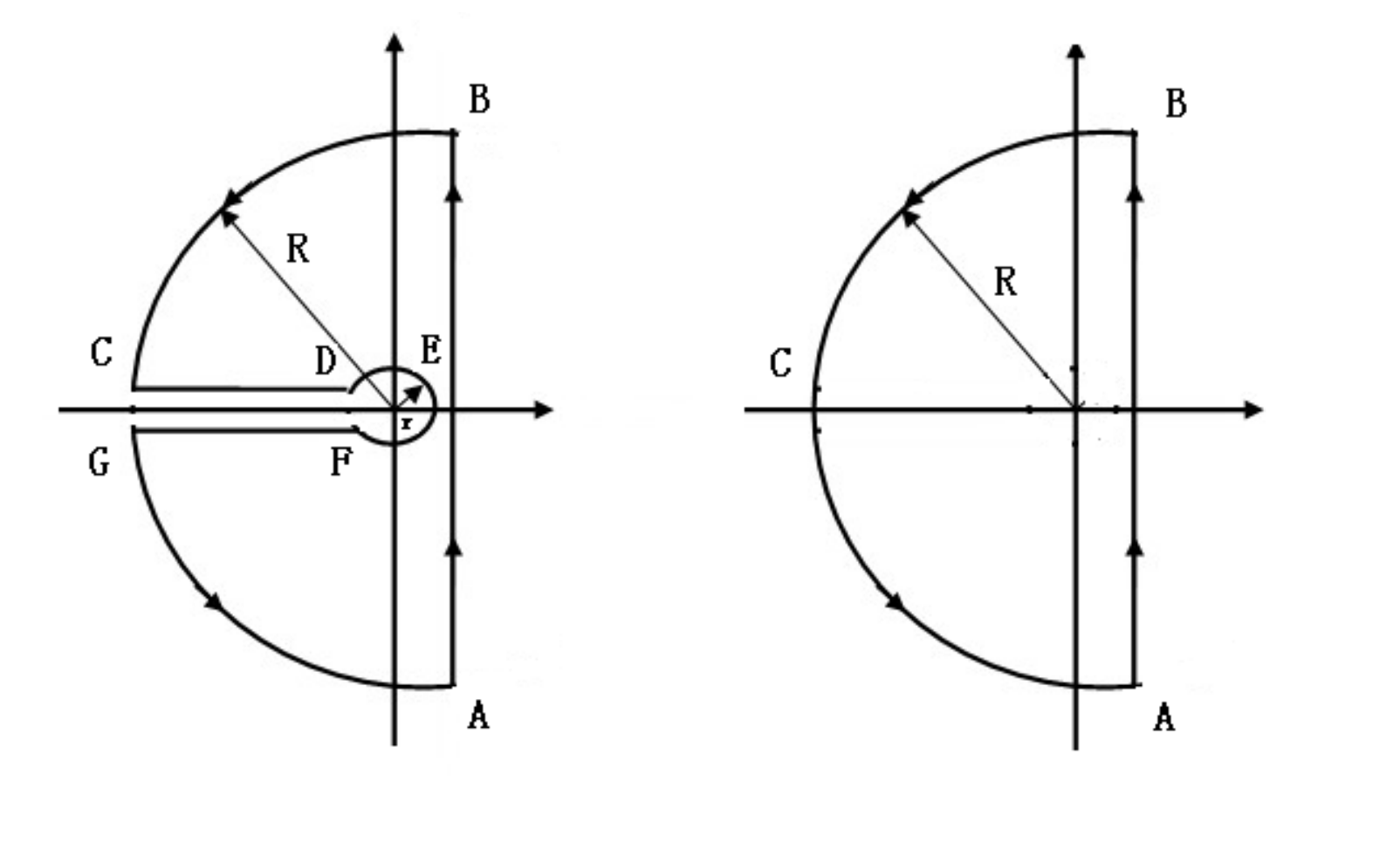}
\end{center}
\vspace*{-0.7cm} \caption{\small Contour $L$ for the
inverse Laplace transform. Left: $d=2$. Right:
$d=3.$}\label{contour1}
\end{figure}

We know from  Lemma \ref{lemma2} that $F_n(z)$ has a finite number
of simple poles in the second and third quadrants, Therefore, for
any $n\ge 0,$ it follows from the residue theorem that
\begin{equation}\label{residathm}
\begin{split}
&\lim\limits_{\scriptstyle R \to +\infty \hfill \atop
\scriptstyle r \to 0^+ \hfill}\oint_L F_n(z)e^{czt/b}dz =
2\pi \ri \sum\limits_{j=1}^{M_n}{\rm Res}\,\big[F_n(z) e^{czt/b}, z_{j}^n\big]\\
&\quad =2\pi \ri \sum\limits_{j=1}^{M_n} \lim\limits_{z\rightarrow
z_{j}^n}\bigg\{(z-z^{n}_j)e^{czt/b}\left[z+\frac{1}{2}+z\frac{K'_n(z)}{K_n(z)}\right]\bigg\}
=2\pi \ri \sum\limits_{j=1}^{M_n} e^{ct z_{j}^{n}/b}z_j^n.
\end{split}
\end{equation}
Thus, by \eqref{inverlap}  and \eqref{residathm},
\begin{equation}\label{Knzformu}
\begin{split}
\frac{2b^2\pi \ri} c \sigma_n(t)&=2\pi \ri \sum\limits_{j=1}^{M_n}
e^{ct z_{j}^{n}/b}z_j^n - \lim\limits_{\scriptstyle R \to +\infty
\hfill \atop\scriptstyle r \to 0^+
\hfill}\left[\int_{\wideparen{BC}}F_n(z)e^{czt/b}dz+\int_{\overline{CD}}
F_n(z)e^{czt/b}dz\right. \\
&\left.+\int_{\wideparen{DEF}}F_n(z)e^{czt/b}dz+\int_{\overline{FG}}F_n(z)e^{czt/b}dz
+\int_{\wideparen{GA}}F_n(z)e^{czt/b}dz\right].
\end{split}
\end{equation}
In view of \eqref{akasymp}, we find from the Jordan's lemma (cf. \cite{DaviesBrian2002,CohenLap2007}) and a direct calculation that
\begin{equation}\label{tempw2}
\lim\limits_{\scriptstyle R \to +\infty \hfill
\atop\scriptstyle r \to 0^+
\hfill}\left[\int_{\wideparen{BC}}F_n(z)e^{czt/b}dz+
\int_{\wideparen{DEF}}F_n(z)e^{czt/b}dz+\int_{\wideparen{GA}}F_n(z)e^{czt/b}dz\right]=0,
\end{equation}
since each contour integral tends to zero. Thus, we only need to
evaluate the integrals along the line segments $\overline{CD}$ and
$\overline{FG}$. We have
\begin{equation*}\label{egnRr}
\begin{split}
\lim\limits_{\scriptstyle R \to +\infty \hfill \atop \scriptstyle r
\to 0^+ \hfill}&\left[\int_{\overline{CD}}F_n(z)e^{czt/b}dz
+\int_{\overline{FG}}F_n(z)e^{czt/b}dz\right]
=\int_0^{\infty}re^{-ctr/b}\left[\frac{K'_n(re^{-\ri\pi})}{K_n(re^{-\ri\pi})}-\frac{K'_n(re^{\ri\pi})}{K_n(re^{\ri\pi})}\right]dr.
\end{split}
\end{equation*}
By  Formula  9.6.31 of \cite{Abr.S84}, 
\begin{equation*}\label{14}
 K_{n}(re^{\pi \ri})=e^{-n\pi \ri}K_{n}(r)-\pi \ri I_{n}(r),\quad
K_{n}(re^{-\pi \ri})=e^{n\pi \ri}K_{n}(r)+\pi \ri I_{n}(r),
\end{equation*}
which, together with \eqref{asympta}, implies
\begin{equation}\label{12www}
\begin{split}
&\frac{K'_n(re^{\ri\pi})}{K_n(re^{\ri\pi})}-\frac{K'_n(re^{-\ri\pi})}{K_n(re^{-\ri\pi})}=\frac{K_{n+1}(re^{-\ri\pi})}
{K_n(re^{-\ri\pi})}-\frac{K_{n+1}(re^{\ri\pi})}{K_n(re^{\ri\pi})}\\
&\quad =\frac{e^{(n+1)\pi \ri}K_{n+1}(r)+\pi \ri I_{n+1}(r)}{e^{n\pi
\ri}K_{n}(r)+\pi \ri I_{n}(r)}
-\frac{e^{-(n+1)\pi \ri}K_{n+1}(r)-\pi \ri I_{n+1}(r)}{e^{-n\pi \ri}K_{n}(r)-\pi \ri I_{n}(r)}\\
&\quad =\frac{(-1)^n 2\pi \ri [K_{n+1}(r)I_n(r)+
K_n(r)I_{n+1}(r)]}{K^2_{n}(r)+\pi^2I^2_{n}(r)}= \frac{(-1)^{n} 2\pi \ri }{r\big[ K^2_{n}(r)+\pi^2I^2_{n}(r)\big]},
\end{split}
\end{equation}
where  we used  the Wronskian identity (see Formula 9.6.15 of \cite{Abr.S84}):
\begin{eqnarray}\label{wronska}
I_{n}(z)K_{n+1}(z)+I_{n+1}(z)K_{n}(z)=z^{-1}.
\end{eqnarray}
A combination of \eqref{residathm}-\eqref{12www} leads to
\begin{equation}\label{finalassd}
\sigma_n(t)=\frac{c}{b^2}\Big\{\sum\limits_{j=1}^{M_n}z_j^n e^{ct
z_j^n/b}+ \int_0^{\infty} \frac{(-1)^ne^{-c
tr/b}}{K^2_{n}(r)+\pi^2I^2_{n}(r)}dr \Big\},
\end{equation}
which is  the  expression \eqref{sigmnt2d}.

Now, we turn to the three dimensional case.  It is important to
notice that the kernel function $F_{n+1/2}(z)$ is not a multi-valued
complex function, as opposed to the two dimensional case. Indeed,
although $K_{n+{1}/{2}}(z)$ is multi-valued, in view of the formula
(see Page 80 of \cite{watson}):
\begin{eqnarray}
K_{n+{1}/{2}}(z)=\sqrt{\frac{\pi}{2z}}\sum\limits_{k=0}^n\frac{(n+k)!e^{-z}}{k!(n-k)!(2z)^k},\quad \forall n\ge 0,
\end{eqnarray}
 the fact $1/\sqrt{z}$ can be eliminated from  the ratio $K_{n+1/2}'/K_{n+1/2}.$
Thanks to this property, we use the contour in Figure \ref{contour1}
(right), by the residue theorem and Jordan's lemma, we have
\begin{equation}\label{Fznp1a}
\begin{split}
\lim_{R\to+\infty}&\oint_L F_{n+1/2}(z)e^{czt/b}dz =
\int_{\gamma-\infty \ri}^{\gamma+\infty
\ri}F_{n+1/2}(z) e^{czt/b} dz\\
&= 2\pi \ri \sum\limits_{j=1}^{M_\nu}{\rm Res}\,\big[F_{n+1/2}(z) e^{czt/b}, z_{j}^\nu\big]
=2\pi \ri \sum\limits_{j=1}^{M_\nu} e^{ct z_{j}^{\nu}/b}z_j^\nu,
\end{split}
\end{equation}
for $\nu=n+1/2.$
This leads to the formula \eqref{sigmnt3d}.  \qed

\vskip 10pt

\section{Justification for Remark \ref{ksweiomegs}}\label{justappB}

Thanks to \eqref{khrela}, we have
\begin{equation*}
z\frac{H_{\nu}^{(1)'}(z)}{H_{\nu}^{(1)}(z)} =-\ri
z\frac{K_{\nu}^{'}(-\ri z)}{K_{\nu}(-\ri z)}.
\end{equation*}
Thus, by \eqref{znuschform} with $\nu=n$ and $-\ri z=bs/c,$ we find
\begin{equation*}\label{B.85}
\begin{split}
\frac{s}{c}&+\frac{1}{2b}+\frac{s}{c}\frac{K_n'(bs/c)}{K_n(bs/c)}=-\frac{\ri
c}{b^2} \sum\limits_{j=1}^{N_{n}}\frac{h_{n,j}} {s+\frac{\ri
ch_{n,j}}{b}}+\frac{c}{b^2}\int_0^{\infty}\frac{(-1)^n}{K_n^2(r)
+\pi^2 I_{n}^2(r)}\frac{1}{\frac{cr}{b}+s}dr.
\end{split}
\end{equation*}
Applying  the inverse Laplace transform to both sides of the above identity leads to
\begin{equation*}
\begin{split}
\mathcal{L}^{-1}\left[\frac{s}{c}+\frac{1}{2b}+\frac{s}{c}\frac{K_n'(bs/c)}{K_n(bs/c)}\right]
=\frac{c}{b^2}\sum\limits_{j=1}^{N_{n}}(-\ri h_{n,j})e^{-{\ri
cth_{n,j}}/{b}}+\frac{c}{b^2}\int_0^{\infty}\frac{(-1)^n
e^{-{crt}/{b}}}{K_n^2(r) +\pi^2 I_{n}^2(r)}dr,
\end{split}
\end{equation*}
which turns out to be \eqref{finalassd}, since  $-\ri h_{n,j}$ are
zeros of $K_n(z)$ and $N_n=M_n.$ \qed

\vskip 10pt

\section{Proof of Proposition \ref{exact}}\label{justappC}
\renewcommand{\theequation}{C.\arabic{equation}}
\setcounter{equation}{0}

This problem can be solved by Laplace transform and separation of
variables. Indeed, the Fourier coefficients $\{ \widehat U_n\}$ in
\eqref{exactsln} can be expressed as
\begin{equation}\label{appC2}
 \widehat U_n(r,t)={\mathcal
 L}^{-1}\Big(\frac{K_n(sr/c)}{K_n(sb_0/c)}\Big)(t)\ast \widehat{G}_n(t),\quad r>b_0.
\end{equation}
We next sketch the evaluation of the inverse Laplace transform by using the residue theorem as in Appendix \ref{pffssl}
for $\sigma_n(t).$ Using \eqref{asympt}, we find $\frac{K_n(sr/c)}{K_n(sb_0/c)}\sim
\sqrt{\frac{b_0}{r}}e^{-s\beta_0},$ where $\beta_0=(r-b_0)/c.$ In order to use the Jordan's lemma (cf. \cite{DaviesBrian2002,CohenLap2007}),
we write
\begin{equation}\label{appC3}
\begin{split}
&{\mathcal
L}^{-1}\Big(\frac{K_n(sr/c)}{K_n(sb_0/c)}\Big)(t)={\mathcal
L}^{-1}\bigg(e^{-\beta_0 s}\Big\{e^{\beta_0
s}\frac{K_n(sr/c)}{K_n(sb_0/c)}-\sqrt{\frac{b_0}{r}}\Big\}\bigg)+ {\mathcal
L}^{-1}\Big(e^{-\beta_0 s} \sqrt{\frac{b_0}{r}}\Big)(t)\\
&\qquad =\widetilde{H}_n(r,
t-\beta_0)U_{\beta_0}(t)+\sqrt{\frac{b_0}{r}}\delta(t-\beta_0),
\end{split}
\end{equation}
where $\widetilde{H}_n(r,t)={\mathcal L}^{-1}\big(e^{\beta_0
s}\frac{K_n(sr/c)}{K_n(sb_0/c)}-\sqrt{{b_0}/{r}}\big)(t),$
$U_{\beta_0}(t)$ is the unit step function (which takes $1$ for
$t\ge \beta_0,$ and $0$ for $t<\beta_0$), and $\delta(t)$ is the
Dirac delta function.  By  applying the residue theorem and Jordan's
lemma to the Bromwich's contour integral:
\begin{equation}\label{appC4}
\begin{split}
\widetilde{H}_n(r,t)=\frac 1 {2\pi \ri }\int_{\gamma-\infty
\ri}^{\gamma +\infty \ri} \Big(e^{\beta_0
s}\frac{K_n(sr/c)}{K_n(sb_0/c)}-\sqrt{\frac{b_0}{r}}\Big)e^{ts} ds,
\end{split}
\end{equation}
we obtain \eqref{exactslncoef1} by  using the  contour in Figure
\ref{contour1} (left) and the same argument as in Appendix
\ref{pffssl}.  Denoting $H_n(r,t):=\widetilde{H}_n(r,t-\beta_0),$
the expression \eqref{exactslncoef} follows from \eqref{appC3}. We
leave the details to the interested readers.  \qed
\end{appendix}

\baselineskip 10pt
\bibliography{refwave}

\end{document}